\documentclass[11pt, a4paper]{article} 

   \usepackage{pdfsync}

\usepackage{amssymb,amsmath, wasysym, graphicx, epsfig, setspace, wrapfig, comment, pgfplots, theorem}
\usepackage{hyperref}
\usepackage{latexsym,dsfont,amsfonts}
\usepackage{epsfig,psfrag,color}
\usepackage{enumerate}
\usepackage{enumitem}
\usepackage{algpseudocode,algorithm} 
\algdef{SE}[DOWHILE]{Do}{doWhile}{\algorithmicdo}[1]{\algorithmicwhile\ #1}

\usepackage{caption}
\usepackage{subcaption}
\usepackage{mathtools}

\newcommand{\setu}{{\mathrm{\mathfrak{u}}}}

\newcommand{\pd}[3]{\frac{\partial ^{#1} #2}{\partial #3}}

\newcommand{\bigO}{\mathcal{O}}
\newcommand{\eps}{\epsilon}

\DeclareMathOperator{\supp}{supp}

\DeclareMathOperator{\ind}{ind}

\DeclareMathOperator{\dottimes}{\mathbin{{.}{\ast}}}

\newcommand{\satop}[2]{\stackrel{\scriptstyle{#1}}{\scriptstyle{#2}}}

\newcommand{\N}[0]{\mathbb{N}}
\newcommand{\Z}[0]{\mathbb{Z}}
\newcommand{\R}[0]{\mathbb{R}}

\newcommand{\bszero}{{\boldsymbol{0}}}
\newcommand{\bsone}{{\boldsymbol{1}}}

\newcommand{\bsb}{{\boldsymbol{b}}}
\newcommand{\bsc}{{\boldsymbol{c}}}

\newcommand{\bse}{{\boldsymbol{e}}}

\newcommand{\bsq}{{\boldsymbol{q}}}

\newcommand{\bsv}{{\boldsymbol{v}}}
\newcommand{\bsw}{{\boldsymbol{w}}}
\newcommand{\bsx}{{\boldsymbol{x}}}
\newcommand{\bsy}{{\boldsymbol{y}}}
\newcommand{\bsz}{{\boldsymbol{z}}}

\newcommand{\bsalpha}{{\boldsymbol{\alpha}}}

\newcommand{\bsgamma}{{\boldsymbol{\gamma}}}
\newcommand{\bseta}{{\boldsymbol{\eta}}}

\newcommand{\bsnu}{{\boldsymbol{\nu}}}

\newcommand{\bsrho}{\boldsymbol{\rho}}

\newcommand{\bspsi}{\boldsymbol{\psi}}

\newcommand{\bsDelta}{{\boldsymbol{\Delta}}}

\newcommand{\bsPhi}{\boldsymbol{\Phi}}

\newcommand{\calA}{\mathcal{A}}

\newcommand{\calF}{\mathcal{F}}
\newcommand{\calG}{\mathcal{G}}
\newcommand{\calH}{\mathcal{H}}

\newcommand{\calO}{\mathcal{O}}

\newcommand{\rd}{\mathrm{d}}
\newcommand{\rmN}{\mathrm{N}}

\newcommand{\bbP}{\mathbb{P}} 
\newcommand{\bbE}{\mathbb{E}}
\newcommand{\bbR}{\mathbb{R}}

\newcommand{\bbN}{\mathbb{N}}

\newcommand{\cbar}[0]{\overline{c}}
\newcommand{\ubar}[0]{\overline{u}}
\newcommand{\ellbar}[0]{\overline{\ell}}
\newcommand{\phibar}[0]{\overline{\varphi}}

\newcommand{\uhat}[0]{\widehat{u}}

\newcommand{\amin}[0]{a_{\mathrm{min}}}
\newcommand{\amax}[0]{a_{\mathrm{max}}}

\definecolor{darkred}{RGB}{139,0,0}
\definecolor{darkgreen}{RGB}{0,100,0}
\definecolor{darkmagenta}{RGB}{170,0,120}
\definecolor{darkpurple}{RGB}{110,0,180}
\definecolor{darkblue}{RGB}{40,0,200}
\definecolor{darkbrown}{rgb}{0.75,0.40,0.15}

\newtheorem{theorem}{Theorem}
\newtheorem{lemma}[theorem]{Lemma}

\newtheorem{corollary}[theorem]{Corollary}

\newenvironment{proof}{\begin{trivlist}
    \item[\hskip\labelsep{\bf Proof.}]}{$\hfill\Box$\end{trivlist}}
\newtheorem{assumption}{Assumption}

\theoremstyle{plain} \theorembodyfont{\rmfamily}
\newtheorem{remark}[theorem]{Remark}

\newcommand{\genvec}{\boldsymbol{\bsz}_\mathrm{gen}}

\numberwithin{equation}{section}

\graphicspath{{.}{./Figures/}{./PDE/}}

\makeatletter
\let\@fnsymbol\@arabic
\makeatother

\title{Density estimation for elliptic PDE with random input
by preintegration and quasi-Monte Carlo methods}
\date{\today}
\author{Alexander D. Gilbert\footnotemark[1] \and
             Frances Y. Kuo\footnotemark[1] \and
	     Abirami Srikumar\footnotemark[1]
	     }

\begin{document}
\maketitle

\footnotetext[1]{School of Mathematics and Statistics, UNSW Sydney, Sydney NSW 2052, Australia.\\
                           \texttt{alexander.gilbert@unsw.edu.au},\;
                           \texttt{f.kuo@unsw.edu.au},
                           \texttt{a.srikumar@student.unsw.edu.au}
                           }

\begin{abstract}
In this paper, we apply quasi-Monte Carlo (QMC) methods with an initial preintegration step to estimate cumulative distribution functions and probability density functions in uncertainty quantification (UQ). The distribution and density functions correspond to a quantity of interest involving the solution to an elliptic partial differential equation (PDE) with a lognormally distributed coefficient and a normally distributed source term. 
There is extensive previous work on using QMC to compute expected values in UQ, which have proven very successful in tackling a range of different PDE problems.
However, the use of QMC for density estimation applied to UQ problems will be explored here for the first time. Density estimation presents a more difficult challenge compared to computing the expected value due to discontinuities present in the integral formulations of both the distribution and density. Our strategy is to use preintegration to eliminate the discontinuity by integrating out a carefully selected random parameter, so that QMC can be used to approximate the remaining integral.  First, we establish regularity results for the PDE quantity of interest that
are required for smoothing by preintegration to be effective.
We then show that an $N$-point lattice rule can be constructed for the integrands corresponding to the distribution and density, such that after preintegration the QMC error is of order $\calO(N^{-1+\epsilon})$ for arbitrarily small $\epsilon>0$. This is the same rate achieved for computing the expected value of the quantity of interest. Numerical results are presented to reaffirm our theory.

\end{abstract}

\section{Introduction}\label{sec:Intro}

Quasi-Monte Carlo (QMC) methods have recently had great success in
tackling high-dimensional problems in uncertainty quantification (UQ).
While most of the research has focussed on computing the expected value of
some quantity of interest involving the solution to a partial
differential equation (PDE) with random input, see, e.g.,
\cite{GanKS21,GGrKSS19,GrKNSSS15,GrKNSS11,HerSchw19,Kaz19,KSS12},
the task of computing other statistical objects for the quantity of
interest has received less
attention. In this paper, we
introduce a method to approximate efficiently the \emph{cumulative
distribution function} (cdf) and \emph{probability density function} (pdf)
of the quantity of interest for an elliptic PDE with lognormally and normally distributed random
inputs. 

Let $D \subset \R^d$, for $d = 1, 2$ or $3$, be a bounded convex domain and
let $(\Omega, \calF, {\bbP})$ be a probability space.
Consider the
stochastic elliptic PDE
	\begin{align}
		\label{eq:pde_0}
		-\nabla \cdot (a(\bsx, \omega)\nabla u(\bsx, \omega)) \,&=\, \ell(\bsx, \omega)\,,
		&\bsx \in D\,,
		\\\nonumber
		u(\bsx, \omega)\,&=\, 0\,,
		&\bsx \in \partial D\,,
	\end{align}
where the divergence and gradient are with respect to the \emph{physical} variable $\bsx
\in D$, and where $\omega \in \Omega$ is a \emph{random} variable
modelling the uncertainty in both the diffusion coefficient $a$ and the
source term $\ell$. The random variable $\omega$ modelling the uncertainty
will be specified formally below, but for now we note that it will
typically be represented by a high number of parameters and that the
randomness in the coefficient and the source term are assumed to be
independent.

Our quantity of interest is a random variable $X \in \R$ given by a bounded linear functional $\calG$
of the solution to the PDE \eqref{eq:pde_0},
\begin{equation}
\label{eq:X}
X  \,=\, \calG(u(\cdot, \omega))\,.
\end{equation}
The cdf and pdf of $X$ are denoted by $F$ and $f$, respectively, and for
$t \in [t_0, t_1] \subset\bbR$,
they can be formulated as the expected values
\begin{align}
\label{eq:cdf0}
F(t) \,&=\, \bbP[X \leq t] \,=\, \bbE_\omega [\ind(t - X)]
\,=\, \int_\Omega \ind(t - \calG(u(\cdot, \omega))) \, \rd \bbP(\omega),
\quad \text{and}\\
\label{eq:pdf0}
f(t) \,&=\, \bbE_\omega [\delta(t - X)]
\,=\,\int_\Omega \delta(t - \calG(u(\cdot, \omega))) \, \rd \bbP(\omega),
\end{align}
where $\ind$ is the indicator function and
$\delta$ is the Dirac $\delta$ distribution, defined by
\[
\delta(z) \,=\, 0 \quad \text{for all } z \neq 0
\quad \text{and} \quad
\int_{-\infty}^\infty \phi(y) \delta(y) \, \rd y \,=\, \phi(0)
\]
for all sufficiently smooth functions $\phi$.

The strategy for approximating the cdf $F$ and the pdf $f$ is to first use
preintegration to smooth the discontinuity in each integrand,
induced by the indicator and Dirac $\delta$ functions, respectively,
and then to approximate the remaining integral
by a QMC rule.

Motivated by the case where the diffusion coefficient and the source
term are given by \emph{independent} Gaussian random fields with separate
Karhunen--Lo\'eve expansions, we assume that $a$ is a \emph{lognormal}
coefficient and $\ell$ admits a finite affine series expansion. Let $s \in \N$
and let $\bsw = (w_0, w_1, \ldots, w_s) \in \R^{s + 1}$ be a vector
of i.i.d.~normally distributed random variables, i.e., $w_i \sim
\rmN(0, 1)$. We assume that the source term is given by
\begin{align} \label{eq:lxw}
 \ell(\bsx, \omega) \,\equiv\, \ell(\bsx, \bsw)
 \,=\, \ellbar(\bsx) + \sum_{i = 0}^s w_i\, \ell_i(\bsx),
\end{align}
where $\ellbar, \ell_i \in L^2(D)$ are known and deterministic. The condition that $\ellbar,\ell_i \in L^2(D)$ will be relaxed later. Similarly,
let $\bsz = (z_1, z_2,\ldots, z_s) \in \R^s$ be a vector of
i.i.d.~normally distributed random variables, i.e., $z_j \sim \rmN(0, 1)$.
We assume that the coefficient $a$ is given by
\begin{align} \label{eq:axz}
 a(\bsx, \omega) \,\equiv\, a(\bsx, \bsz)
 \,=\, 
 \exp\bigg(
 \sum_{j = 1}^s z_j\, a_j(\bsx)\bigg),
\end{align}
where
$a_j \in L^\infty(D)$
are known and deterministic. However, the requirements on $a_j$ will be strengthened as necessary later in the analysis. In this formulation, the coefficient and the
source term are independent and the random variable $\omega$ is given as a
collection of normally distributed parameters $w_0, w_1, \ldots, w_s$ and
$z_1, \ldots, z_s$. Note that $w_i$ is indexed from $0$ to $s$ while
$z_j$ is indexed from $1$ to $s$, for the sake of convenience when we
discuss preintegration below.

To simplify the notation, we write the parameters in a single random
vector $\bsy \in \R^{2s + 1}$ with
\[
 \bsy = (y_0, y_1, \ldots, y_s, y_{s + 1}, \ldots, y_{2s})
 = (w_0, w_1, \ldots, w_s, z_1, \ldots, z_s) = (\bsw,\bsz) \in \bbR^{2s+1}.
\]
We also write the PDE solution in terms of $\bsy$ as $u(\bsx, \bsy)$ and
then we write the quantity of interest \eqref{eq:X} as
\begin{align} \label{eq:qoi}
  X = \calG(u(\cdot,\bsy)).
\end{align}
In this way, the expected values in the cdf \eqref{eq:cdf0} and pdf
\eqref{eq:pdf0} can be written as $(2s + 1)$-dimensional integrals with
respect to a product Gaussian density,
\begin{align}
\label{eq:cdf-y0}
F(t) \,&=\, \int_{\R^{2s + 1}} \ind(t - \calG(u(\cdot, \bsy)))\,
\bigg(\prod_{i = 0}^{2s} \rho(y_i) \bigg)\, \rd \bsy,
\quad \text{and} \\
\label{eq:pdf-y0}
f(t) \,&=\, \int_{\R^{2s + 1}} \delta(t - \calG(u(\cdot, \bsy))) \,
\bigg(\prod_{i= 0}^{2s} \rho(y_i)\bigg)\, \rd \bsy.
\end{align}
where $\rho(y) = \tfrac{1}{\sqrt{2\pi}}e^{-y^2/2}$ denotes the standard normal density.

The preintegration strategy is to first integrate out the variable $y_0 =
w_0$, so that evaluating the cdf (or pdf) becomes an integral of a
(hopefully smooth) function over the remaining $2s$ variables $\bsy_{1:2s}
= (\bsw_{1:s},\bsz)$. For convenience, later we will often abuse the
notation and write $\bsy$ (or $\bsw$) to denote the vector without the
initial variable $y_0$ (or $w_0$). There shall be no ambiguity from the
context whether $\bsy$ (or $\bsw$) includes the variable $y_0$ (or $w_0$).
We will make a critical assumption that
\begin{equation*}
  \ell_0(\bsx) > 0 \quad\mbox{for all}\quad \bsx\in D,
\end{equation*}
which ensures that the quantity of interest \eqref{eq:qoi} is strictly monotone with respect to the preintegration variable $y_0$, i.e.,
\begin{equation}
\label{eq:mono}
\frac{\partial X}{\partial y_0} \,=\, \frac{\partial}{\partial y_0} \calG(u(\cdot, \bsy)) \,>\, 0
\quad \text{for all } \bsy \in \R^{2s + 1}.
\end{equation}
This \emph{monotonicity property} \eqref{eq:mono} is necessary for the preintegration theory (to be explained later in the article).

Preintegration is a special case of \emph{conditional Monte Carlo}, or
\emph{conditional sampling} when more general quadrature rules are used,
and as a numerical method it applies to more general non-smooth integrands
than those in \eqref{eq:cdf-y0}, \eqref{eq:pdf-y0}. Conditional sampling
is more general in that it allows one to condition on (or integrate out)
partial information that is not restricted to a single variable. It has
been used extensively as a smoothing and variance reduction technique in
statistics and computational finance
\cite{ACN13a,ACN13b,BayB-HTemp22,Glasserman,GlaSta01,Hol11,LecLem00,LiuOw23,WWH17}.
The smoothing effect of preintegration was analysed theoretically in
\cite{GKLS18}, which builds upon the prior work \cite{GKS10,
GKS13,GKS17note}. More recently, the use of conditional sampling for
density estimation was introduced in \cite{LecPucBAb22} and the special
case of preintegration for density estimation was analysed in
\cite{GKS23}, with the latter paper \cite{GKS23} providing the setting
followed here. A key requirement for the smoothing by preintegration theory
from \cite{GKS23,GKLS18} is the monotonicity condition \eqref{eq:mono},
which was shown in \cite{GKS22b} to be necessary in the sense that if it does not hold then 
after preintegration there may still exist singularities.

There are several natural questions about how the methodology and analysis
in this paper may be extended, all of which essentially come back to 
the requirements for the preintegration theory and, in particular, the monotonicity condition \eqref{eq:mono}.
First, the method in this paper can be extended to 
a nonlinear function of the source term $\ell(\bsx, \omega)$ \eqref{eq:lxw},
provided that the function is monotone so that \eqref{eq:mono} continues to hold.
Second, if the source term is deterministic then the general idea of the method can be extended
by applying preintegration to a variable in the coefficient \eqref{eq:axz}.
However, in this case the monotonicity condition \eqref{eq:mono} would need to be proved
in order for the smoothing by preintegration theory to apply.
This condition would need to be shown to hold on a case-by-base basis depending
on the coefficient and the quantity of interest, and it
is an interesting direction for future research.
For example, if $a_1(\bsx)$ in \eqref{eq:axz} is constant, then it can be shown that the quantity of interest is
monotone with respect to $z_1$.
Note that by dividing the PDE \eqref{eq:pde_0} through 
by $e^{a_1 z_1}$ this example fits the setting described above of a monotone nonlinear source term.
Third, motivated by applications to modelling fluid flow through porous media where Gaussian
random fields are used to represent the uncertain properties of the medium, we
have considered a source term and coefficient of the form \eqref{eq:lxw}--\eqref{eq:axz},
depending on normally-distributed random variables $y_j$.
Our method can be extended to other stochastic models for the uncertain parameters,
provided that each random variable $y_j$ has full support on $\R$ and its density is sufficiently
smooth (so that the preintegration theory holds, see \cite{GKS23,GKLS18}).
However, our method \emph{cannot} be extended to the simpler affine uniform case as in, e.g., \cite{KSS12},
because in this case each $y_j \in [0, 1]$ is bounded and for bounded domains
a single preintegration step is not sufficient to recover the necessary smoothness.
For further details see \cite{GKS10} and Appendix A2 in \cite{GKS23}.
For further details on the first two points above see Remark~\ref{rem:extensions} below.

In practice, to evaluate the quantity of interest \eqref{eq:qoi} the PDE
\eqref{eq:pde_0} must be solved numerically by discretising the spatial
domain, e.g., by a finite element method. Preintegration 
for the discrete problem requires a considerable amount of additional theory so will be discussed in depth in a subsequent paper \cite{Gil24}. However, for our
numerical results we still use piecewise linear finite elements.

The structure of this paper is as follows: In Section~\ref{sec:preint}, we
introduce the background theory on smoothing by preintegration, including
the function space setting, preintegration for density estimation and QMC
methods for high-dimensional integration. In Section~\ref{sec:pde},
we introduce the PDE with random input, along with our key assumptions and
some important properties that will be used throughout the paper.
Section~\ref{sec:de4pde} introduces our method for approximating the cdf
and pdf of a quantity of interest coming from a PDE with random input.
Then in Section~\ref{sec:cond_preint}, we verify that the PDE quantity of
interest satisfies the conditions required for the smoothing by
preintegration theory to apply. In particular, we prove that \eqref{eq:mono} holds.
Following this, a full error analysis
of our method is presented in Section~\ref{sec:error}.
Section~\ref{sec:num} presents numerical results that support the
theoretical results from our error analysis and Section~\ref{sec:conc} provides concluding remarks.

\section{Smoothing by preintegration}
\label{sec:preint}

In this section we summarise the results from \cite{GKS23} for estimating
the cdf and pdf of a real-valued random variable $X = \varphi (\bsy)$,
which are given by
\begin{align}
\label{eq:cdf-y}
F(t) \,&=\, \int_{\R^{2s + 1}} \ind(t - \varphi(\bsy))\,
\bigg(\prod_{i = 0}^{2s} \rho(y_i) \bigg)\, \rd \bsy
\quad \text{and} \\
\label{eq:pdf-y}
f(t) \,&=\, \int_{\R^{2s + 1}} \delta(t - \varphi(\bsy))\,
\bigg(\prod_{i= 0}^{2s} \rho(y_i)\bigg)\, \rd \bsy,
\end{align}
for a generic function $\varphi$ (ignoring the specific PDE problem for
now), where the components $y_i$ are i.i.d.\ normally distributed
random variables. The strategy is to preintegrate the non-smooth
integrands in \eqref{eq:cdf-y} and \eqref{eq:pdf-y} with respect to $y_0$
to obtain functions in $2s$ variables $\bsy = (y_1,\ldots,y_{2s}) \in
\bbR^{2s}$:
\begin{align}
  g_{\rm cdf}(\bsy) &\,\coloneqq\, \int_{-\infty}^\infty \ind(t - \varphi(y_0,\bsy)) \rho(y_0)\,\rd y_0\,, \label{eq:g-cdf}
  \quad\mbox{and} \\
  g_{\rm pdf}(\bsy) &\,\coloneqq\, \int_{-\infty}^\infty \delta(t - \varphi(y_0,\bsy)) \rho(y_0)\,\rd y_0\,. \label{eq:g-pdf}
\end{align}
Then the cdf \eqref{eq:cdf-y} and pdf \eqref{eq:pdf-y} become
$2s$-dimensional integrals of the functions $g_{\rm cdf}$ and $g_{\rm
pdf}$,
\begin{equation}
\label{eq:preint-cdfpdf}
  F(t) = \int_{\bbR^{2s}} g_{\rm cdf}(\bsy)\bigg(\prod_{i = 1}^{2s} \rho(y_i)\bigg)\, \rd \bsy
  \quad\mbox{and}\quad
  f(t) = \int_{\bbR^{2s}} g_{\rm pdf}(\bsy)\bigg(\prod_{i = 1}^{2s} \rho(y_i)\bigg)\, \rd \bsy.
\end{equation}
It is proved in \cite{GKS23} that $g_{\rm cdf}$ \eqref{eq:g-cdf} and
$g_{\rm pdf}$ \eqref{eq:g-pdf} will be sufficiently smooth under
appropriate conditions, and therefore the $2s$-dimensional integrals can
be tackled successfully by QMC rules.

We start by introducing the function spaces needed to quantify smoothness
and then summarise the results from \cite{GKS23}.

\subsection{Function spaces} \label{sec:function_spaces}

We only define the $(2s + 1)$-variate spaces, since the
$2s$-variate spaces can be defined analogously by simply excluding the
variable $y_0$.

Let $\bbN_0 \coloneqq \{0,1,2,\ldots\}$. For $i = 0, 1, \ldots, 2s$ and a
multi-index $\bsnu \in \N_0^{2s + 1}$, let
\[
 \partial^i \coloneqq \pd{}{}{y_i}
 \qquad\mbox{and}\qquad
 \partial^\bsnu \,\coloneqq\, \prod_{i = 0}^{2s} \pd{\nu_i}{}{y_i^{\nu_i}}
\]
be the first-order derivative and the higher-order mixed derivative of
order $\bsnu$, respectively. This notation is used for derivatives
with respect to $\bsw = (y_0, y_1, \ldots, y_s)$ and $\bsz = (y_{s + 1},
y_{s + 2}, \ldots, y_{2s})$ as well. The $\bsnu$th weak derivative of a function
$\phi:\bbR^{2s+1}\to\bbR$ is the distribution $\partial^\bsnu  \phi$ that
satisfies
\begin{align*} 
 \int_{\R^{2s + 1}} & \partial^\bsnu  \phi(\bsy)\, v(\bsy) \,\rd \bsy
 \,=\, (-1)^{|\bsnu|} \int_{\R^{2s + 1}} \phi(\bsy)\, \partial^\bsnu  v(\bsy) \,\rd \bsy
 \quad \text{for all } v \in C^\infty_0(\R^{2s + 1}),
\end{align*}
where $C^\infty_0(\R^{2s + 1})$
is the space of infinitely differentiable functions with compact support.

For $\bsnu \in \N_0^{2s + 1}$, let $C^\bsnu(\R^{2s + 1})$ denote the space
of functions with continuous \emph{mixed} (classical) derivatives up
to order $\bsnu$:
\[
C^\bsnu(\R^{2s + 1}) \coloneqq \big\{\phi \in C(\R^{2s + 1}) :
\partial^{\bseta}  \phi \in C(\R^{2s + 1})
\text{ for all } \bseta \leq \bsnu\}\,,
\]
where $C(\R^{2s + 1})$ denotes the space of continuous functions on
$\R^{2s + 1}$.

Moreover, for $\bsnu \in \N_0^{2s + 1}$ we define the \emph{Sobolev space
of dominating mixed smoothness} of order~$\bsnu$, denoted by
$\calH^\bsnu_{2s + 1}$, to be the space of locally integrable functions on
$\R^{2s + 1}$ such that the norm
\begin{equation*}
\|\phi\|_{\calH^\bsnu_{2s + 1}}^2 \coloneqq
\sum_{\bseta \leq \bsnu} \frac{1}{\gamma_\bseta}
\int_{\R^{2s  +1}} |\partial^{\bseta}  \phi(\bsy)|^2\,
\bspsi_{\bseta}(\bsy_\bseta)\,\bsrho_{-\bseta}(\bsy_{-\bseta})
 \,\rd \bsy
\end{equation*}
is finite, where we introduced the shorthand notations
\[
  \gamma_\bseta \coloneqq \gamma_{\supp(\bseta)}, \quad
  \bspsi_{\bseta}(\bsy_\bseta) \coloneqq \prod_{i=0,\,\eta_i \neq 0}^{2s}
  \psi(y_i)
  \quad\mbox{and}\quad
  \bsrho_{-\bseta}(\bsy_{-\bseta}) \coloneqq \prod_{i=0,\,\eta_i = 0}^{2s}
  \rho(y_i),
\]
where $\supp(\bseta) \coloneqq \{ i\in \{0:2s\} : \eta_i\ne 0\}$ is the set
of indices for the nonzero components of $\bseta$. Here the behaviour of
derivatives as $y_i \to \pm \infty$ is controlled by a strictly positive and
integrable \emph{weight function} $\psi : \R \to \R$.
Also, here $\bsgamma \coloneqq \{\gamma_\setu > 0 : \setu \subseteq
\{0:2s\}\}$ is a collection of positive real numbers called \emph{weight
parameters}; they model the relative importance of different collections
of variables, i.e., $\gamma_\setu$ describes the relative importance of
the collection of variables $(y_i : i \in \setu)$. We set
$\gamma_\emptyset \coloneqq 1$.
We define analogously the $2s$-variate Sobolev space $\calH_{2s}^\bsnu$,
with variables indexed from $1$ to $2s$.

An important property of the Sobolev space of \emph{first-order}
dominating mixed smoothness, i.e., $\calH^{\bsone}_{2s}$ with $\bsone
\coloneqq (1, 1, \ldots, 1)$, is that it is equivalent to the
unanchored space over the unbounded domain $\R^{2s}$ introduced in
\cite{NK14}. Explicitly, it was shown recently in \cite{GKS22a} that if
the weight function $\psi$ satisfies
\begin{equation}
\label{eq:psi}
\int_{-\infty}^\infty \frac{\Phi(y)(1 - \Phi(y))}{\psi(y)} \, \rd y \,< \, \infty,
\end{equation}
where $\Phi$ is the cdf of $\rho$, then $\calH^{\bsone}_{2s}$ and the
unanchored space from \cite{NK14} are equivalent. This equivalence is
crucial as it immediately shows that the bounds on the QMC error from
\cite{NK14} also hold in $\calH^{\bsone}_{2s}$. Examples of common
pairings $(\rho, \psi)$ satisfying \eqref{eq:psi} can be found in
\cite[Table~3]{KSWW10}. Note that for simplicity, $\psi_i^2$ in
\cite{KSWW10,NK14}, and also in \cite{GKS22a,GKS23}, is replaced here by
a single $\psi$. The analysis in this paper will consider two types of
weight functions, exponential and Gaussian. Exponential weight functions
are of the form
\begin{equation}
\label{eq:psi-exp}
\psi(y) \,=\, \exp(-2\, \mu\, |y|) \quad  \text{for }  \mu > 0,
\end{equation}
whereas Gaussian weight functions have the form
\begin{equation}
\label{eq:psi-Gauss}
\psi(y) \,=\, \exp(-  \mu\, y^2)\quad  \text{for }  \mu \in (0, 1/2).
\end{equation}
Both choices of weight functions satisfy the condition \eqref{eq:psi}.

\subsection{Preintegration theory for density estimation}
\label{sec:preint-theory}

In this section we briefly summarise the theory on using smoothing by
preintegration for density estimation from \cite{GKS23}. First, we make
the following assumptions on $\varphi$ and~$\rho$.

\begin{assumption} \label{asm:phi}
For $s\ge 1$ and $\bsnu\in \bbN_0^{2s}$, let $\varphi :
\bbR^{2s+1}\to\bbR$ in \eqref{eq:g-cdf} and \eqref{eq:g-pdf} satisfy
\begin{enumerate}[label=\textnormal{(\alph*)}] 
\item \label{itm:mono} $ \partial^{0} \varphi(y_0, \bsy) > 0$ for all
    $ (y_0, \bsy) \in \R^{2s + 1}$;
\item\label{itm:inf} for each $\bsy \in \R^{2s}$, $\varphi(y_0, \bsy) \to \infty$ as
    $y_0 \to \infty$; and
\item\label{itm:phi-smooth} $\varphi \in \calH^{(\nu_0, \bsnu)}_{2s + 1} \cap C^{(\nu_0,
    \bsnu)}(\R^{2s + 1})$, where $\nu_0 \coloneqq |\bsnu|+1$.
\end{enumerate}
Additionally, suppose that $\rho \in C^{|\bsnu|}(\bbR)$.
\end{assumption}

Assumptions~\ref{asm:phi}\ref{itm:mono}--\ref{itm:inf} imply that
$\varphi$ is strictly increasing with respect to $y_0\in\bbR$ and also
tends to $\infty$ as $y_0\to\infty$. Consequently, for fixed
$\bsy\in\bbR^{2s}$, the discontinuity induced by the indicator function in
\eqref{eq:cdf-y} either occurs at a unique value of $y_0$ or not at all. For
a given $t\in [t_0,t_1]$, we define the set of $\bsy\in\bbR^{2s}$ for
which the discontinuity occurs by
\begin{align} \label{eq:Ut}
  U_t \,\coloneqq\, \{\bsy\in\bbR^{2s} \,:\, \varphi(y_0,\bsy) = t
  \mbox{ for some } y_0 \in \bbR\}.
\end{align}
Then, the implicit function theorem \cite[Theorem~3.1]{GKS23} states
that if the set
\begin{equation}
\label{eq:V}
R \coloneqq \big\{(t,\bsy) \in (t_0, t_1) \times \R^{2s} : \varphi(y_0, \bsy)  = t
\text{ for some } y_0 \in \R \big\} \,\subset\, [t_0, t_1] \times \R^{2s}
\end{equation}
is not empty, then there exists a unique function $\xi \in C^{(\nu_0,
\bsnu)}(\overline{R})$ satisfying
\begin{align} \label{eq:xi-def}
  \varphi(\xi(t,\bsy),\bsy) \,=\, t
  \qquad  \mbox{for all} \quad (t, \bsy) \in \overline{R}.
\end{align}
Thus, given $t$ and $\bsy$, the function $\xi$ determines the unique value of $y_0$ at which the discontinuity in \eqref{eq:cdf-y} occurs. (Note that the set \eqref{eq:V} was denoted by $V$ in \cite{GKS23}.)
Furthermore, for $(t, \bsy) \in R$ the first-order derivatives of $\xi$
are given by
\begin{align*}
  \partial^i \xi(t,\bsy)
  &= - \frac{\partial^{i}\varphi(\xi(t,\bsy),\bsy)}{\partial^{0} \varphi(\xi(t,\bsy),\bsy)}
  \qquad\mbox{for all } i=1,\ldots,2s, \quad \text{and}
  \\
  \frac{\partial}{\partial t} \xi(t,\bsy)
  &= \frac{1}{\partial^{0} \varphi(\xi(t,\bsy),\bsy)}.
\end{align*}

The uniqueness of $\xi$ follows from the monotonicity condition
Assumption~\ref{asm:phi}\ref{itm:mono}. Similarly, because
$\varphi(\cdot,\bsy)$ is increasing, it follows from \eqref{eq:xi-def}
that $\varphi(y_0,\bsy) < t$ if and only if $y_0 < \xi(t,\bsy)$, and so
the preintegrated functions \eqref{eq:g-cdf} and \eqref{eq:g-pdf} can be
written as
\begin{align} \label{eq:g-xi}
  g_{\rm cdf}(\bsy) 
  \,=\,
  \Phi(\xi(t,\bsy))
  \qquad\mbox{and}\qquad
  g_{\rm pdf}(\bsy) \,=\, \frac{\rho(\xi(t,\bsy))}{ \partial^{0} \varphi(\xi(t,\bsy),\bsy)},
\end{align}
for $\bsy\in U_t$, and both functions give $0$ if $\bsy\in
\bbR^{2s}\setminus U_t$.

To establish the smoothness of $g_{\rm cdf}$ and $g_{\rm pdf}$, we need
expressions for the derivatives of $g_{\rm cdf}$ and $g_{\rm pdf}$. 
The proofs of \cite[Theorems~3.2 and~3.3]{GKS23} established that the
mixed derivatives of $g_{\rm cdf}$ and $g_{\rm pdf}$ can be expressed as
finite sums of functions of the form $h_{q, \bseta}(t, \bsy)$ in
\eqref{eq:h-form} below, which arise due to repeated applications of the chain rule. The proofs also obtained bounds on the total number of terms of the form $h_{q, \bseta}(t, \bsy)$ in the sums for each derivative, which together with the constants $B_{q,\bseta}$ in \eqref{eq:h-int} yield bounds on the
$\calH_{2s}^{\bsnu}$-norm of $g_{\rm cdf}$ and $g_{\rm pdf}$.

\begin{assumption} \label{asm:phi-h}
Let $s\ge 1$, $\bsnu\in \bbN_0^{2s}$, $[t_0,t_1]\subset\bbR$, and suppose
that $\varphi$ and $\rho$ in \eqref{eq:g-cdf} and \eqref{eq:g-pdf}
satisfy Assumption~\ref{asm:phi}. Recall the definitions of $U_t$,
$R$ and $\xi$ in \eqref{eq:Ut}, \eqref{eq:V} and \eqref{eq:xi-def},
respectively.
Given $q\in\bbN_0$ 
and $\bseta\le\bsnu$ satisfying $|\bseta|+q \le
|\bsnu|+1$, we consider functions $h_{q,\bseta}: \overline{R} \to \R$
of the form
\begin{align} \label{eq:h-form}
 \hspace{-0.5cm}
 \begin{cases}
 h_{q, \bseta}(t, \bsy) = h_{q,\bseta,(r,\bsalpha,\beta)}(t,\bsy)
 \coloneqq
 \displaystyle\frac{(-1)^r {\rho}^{(\beta)}(\xi(t,\bsy))\,
       \prod_{\ell=1}^{r} \partial^{\bsalpha_\ell}  \varphi(\xi(t,\bsy),\bsy)}
      {[ \partial^{0}\varphi(\xi(t,\bsy),\bsy)]^{r+q}},
\\[4mm]
 \mbox{with $r\in\bbN_0$, $\bsalpha=(\bsalpha_\ell)_{\ell=1}^r$,
 $\bsalpha_\ell\in \N_0^{2s + 1}\!\setminus\!\{\bse_0,\bszero\}$, $\beta \in \N_0$ satisfying}
 \\[1mm]
 r\le 2|\bseta|+q-1,\;
 \alpha_{\ell,0} \le |\bseta|+q, \;
 \beta \leq |\bseta|+q-1,\;
 \beta\bse_0 + \displaystyle\sum_{\ell = 1}^{r} \bsalpha_\ell
 = (r+q-1, \bseta).
 \end{cases}
 \hspace{-2cm}
\end{align}
We assume that all such functions $h_{q,\bseta}$ satisfy
\begin{equation} \label{eq:h-lim}
 \lim_{\bsy \to \partial U_t} h_{q,\bseta}(t, \bsy) = 0
 \quad \text{for all } t \in [t_0, t_1],
\end{equation}
and there is a constant $B_{q,\bseta}$ such that
\begin{equation} \label{eq:h-int}
 \sup_{t \in [t_0, t_1]}
 \int_{U_t} |h_{q,\bseta}(t, \bsy)|^2 \,
 \bspsi_{\bseta}(\bsy_\bseta)\,\bsrho_{-\bseta}(\bsy_{-\bseta})\,\rd \bsy
 \le B_{q,\bseta} \,<\, \infty.
\end{equation}
\end{assumption}

The parameter $\bsnu$ in Assumptions~\ref{asm:phi} and \ref{asm:phi-h} determines the resulting smoothness of the preintegrated functions, while
the parameter $q$ in Assumption~\ref{asm:phi-h} relates to whether the cdf or pdf is being 
considered.
More precisely, under Assumption~\ref{asm:phi} and
Assumption~\ref{asm:phi-h} for $q=0$ and all $\bszero \ne \bseta\le\bsnu$,
the main result in \cite[Theorem~3.2]{GKS23} states that, for $t \in
[t_0, t_1]$, the preintegrated function for the cdf satisfies
\[
  g_{\rm cdf} \in\calH^{\bsnu}_{2s} \cap C^{\bsnu}(\R^{2s}),
\]
with its $\calH^\bsnu_{2s}$-norm bounded uniformly in $t$,
\begin{align} \label{eq:norm-cdf}
 \|g_{\rm cdf}\|_{\calH^\bsnu_{2s}}
 \le \Bigg(1 +
 \sum_{\bszero\ne\bseta \leq \bsnu}
 \frac{\big(8^{|\bseta|-1}(|\bseta|-1)!\big)^2 B_{0, \bseta}}{\gamma_\bseta}\Bigg)^{1/2}
 \,<\, \infty\,.
\end{align}

Similarly, under Assumption~\ref{asm:phi} and Assumption~\ref{asm:phi-h}
for $q=1$ and all $\bseta\le\bsnu$, the result in
\cite[Theorem~3.3]{GKS23} states that, for $t \in [t_0, t_1]$, the
preintegrated function for the pdf satisfies
\[
  g_{\rm pdf} \in \calH^{\bsnu}_{2s} \cap C^{\bsnu}(\R^{2s}),
\]
with its $\calH^\bsnu_{2s}$-norm bounded uniformly in $t$,
\begin{align} \label{eq:norm-pdf}
 \|g_{\rm pdf}\|_{\calH^\bsnu_{2s}}
 \le \Bigg(\sum_{\bseta \leq \bsnu}
 \frac{\big(8^{|\bseta|}|\bseta|!\big)^2 B_{1, \bseta}}{\gamma_\bseta}\Bigg)^{1/2}
 \,<\, \infty\,.
\end{align}

Later in Section~\ref{sec:cond_preint}, we will obtain explicit constants $B_{q,\bseta}$ for
the special case of the PDE quantity of interest described in the Introduction.

\subsection{Quasi-Monte Carlo methods}
\label{sec:qmc} Quasi-Monte Carlo (QMC) methods are a powerful class of
equal-weight quadrature rules that have proven to be effective at
approximating high-dimensional integrals. Originally designed for
integration on the unit cube, they can also be used for integration on
unbounded domains with respect to a product density $\bsrho$, by mapping
back to the unit cube using the inverse cdf transform. An $N$-point
transformed QMC approximation to a $2s$-dimensional integral, as in
\eqref{eq:preint-cdfpdf}, is given by
\begin{equation}
\label{eq:qmc}
I_{2s}(g)\,\coloneqq\,
\int_{\R^{2s}} g(\bsy) \bsrho(\bsy) \, \rd \bsy
\,\approx\, \frac{1}{N} \sum_{n = 0}^{N - 1} g(\bsPhi^{-1}(\bsq_n))
\,\eqqcolon\, Q_{2s, N} (g)\,,
\end{equation}
where the inverse cdf $\bsPhi^{-1}$ is applied componentwise to map the
QMC points, $\{\bsq_n\}_{n = 0}^{N - 1}$,
from the unit cube to $\R^{2s}$. For further details on QMC see, e.g., \cite{DKS13}.

In this paper, we use a simple class of randomised QMC methods called
\emph{randomly shifted lattice rules}.
For a generating vector $\genvec \in \N^{2s}$ and a uniformly
distributed random shift $\bsDelta \in [0, 1)^{2s}$,
the randomly shifted lattice points are given by
\begin{equation*}
\bsq_n \,\coloneqq\,
\bigg\{ \frac{n \genvec}{N} + \bsDelta\bigg\}
\quad \text{for } n = 0, 1, \ldots, N - 1,
\end{equation*}
where $\{\cdot\}$ denotes taking the fractional part componentwise.

Good generating vectors can be constructed using the
\emph{component-by-component} (CBC) construction. For the unbounded
setting as in \eqref{eq:qmc}, the CBC construction was developed in
\cite{NK14}. Theorem~8 in \cite{NK14} shows that under certain conditions
on the weight function $\psi$, the root-mean-square error (RMSE)  of a
CBC-generated randomly shifted lattice rule achieves close to $\bigO(1/N)$
convergence. Below we present their main result adapted to 
$\calH_{2s}^\bsone$ for the choice of exponential and Gaussian weight
functions (see also \cite{GKS22a}).

\begin{theorem} \label{thm:RMSE-bound-g}
For $s, N \in \N$, let $\psi$ be either an exponential or Gaussian weight function
and let $g\in\calH^{\bsone}_{2s}$ with weight parameters  $\bsgamma = (\gamma_{\bseta})_{\bseta\leq\bsone}$.
Then a randomly shifted lattice rule with $N$ points in $2s$ dimensions can be constructed using a CBC algorithm
such that, for $\eps\in (0, 1/2]$  depending on $\psi$, the RMSE satisfies	
\begin{align} \label{eq:qmc-rmse}
			\sqrt{\bbE_\bsDelta \big[| I_{2s}(g)- Q_{2s, N} (g)|^2\big]}
			 \leq
			\bigg(\frac{1}{\phi_{\rm{tot}}(N)}\sum_{\bszero\neq\bseta\leq\bsone}\gamma_\bseta^{1/(2(1-\eps))}\,\varrho(\eps)^{|\bseta|} \bigg)^{1 - \eps}
			\|g\|_{\calH^\bsone_{2s}}\,,
\end{align}
where $\phi_\mathrm{tot}$ is the Euler totient function and $\varrho$ also
depends on $\psi$.

For an exponential weight function, $\psi(y) = e^{-2 \mu |y|}$ with $ \mu > 0$, the bound \eqref{eq:qmc-rmse}
holds for all $\epsilon \in (0, 1/2]$ with
\begin{align}\label{eq:qmc-err-const-exp}
			\varrho(\eps) \,=\, 2\bigg(\frac{4\sqrt{2\pi}\exp(2\mu^2/\epsilon)}{\pi^{2-\epsilon}(2-\epsilon)\epsilon}\bigg)^{1/(2(1-\eps))}
			\,\zeta\bigg(\frac{1 - \eps/2}{1 - \eps}\bigg) \,,
\end{align}
whereas for a Gaussian weight function, $\psi(y) = e^{- \mu y^2}$ with $ \mu \in (0, 1/2)$,
the bound \eqref{eq:qmc-rmse} holds for all $\eps \in (\mu, 1/2]$ with
\begin{align}\label{eq:qmc-err-const-Gauss}
		\varrho(\eps) 
		\,=\, 2\bigg(\dfrac{\sqrt{2\pi}}{\pi^{2-2 \mu}(1- \mu) \mu}\bigg)^{1/(2(1-\eps))}
		\,\zeta\bigg(\frac{1 - \mu}{1 - \eps}\bigg)\,.
\end{align}
\end{theorem}

\begin{proof}
Theorem 8 in \cite{NK14} gives for all $\lambda\in(1/(2r_2),1]$,
\begin{align} \label{eq:qmc-err}
			\sqrt{\bbE_\bsDelta \big[| I_{2s}(g)- Q_{2s, N} (g)|^2\big]}
			 \leq
			\bigg(\frac{1}{\phi_{\rm{tot}}(N)}\sum_{\bszero\neq\bseta\leq\bsone}\gamma_\bseta^{\lambda}
			\,[2C_2^\lambda \zeta(2r_2 \lambda)]^{|\bseta|} \bigg)^{1/(2\lambda)}
			\|g\|_{\calH^\bsone_{2s}}\,,
		\end{align} where $r_2 \in (0, 1)$ and $C_2$ depend on $\psi$.
Note that we have also used the equivalence of the unanchored space
from \cite{NK14} and $\calH^\bsone_{2s}$ from \cite{GKS22a} to write this
in terms of $\calH^\bsone_{2s}$.

Following Example 5 in \cite{KSWW10} (with parameters $\nu = 1$ and $\alpha = 2/\mu$),
for an exponential weight
function we have
for any $\delta \in (0, 1/2)$,
\[
C_2
  \,=\, \dfrac{\sqrt{2\pi}\exp( \mu^2/\delta)}{\pi^{2-2\delta}(1-\delta)\delta}
			\quad \text{and} \quad
			 r_2 = 1-\delta.
\]
The result \eqref{eq:qmc-rmse} with  \eqref{eq:qmc-err-const-exp} follows by letting  $\varrho(\eps) = 2C_2^\lambda \zeta(2r_2 \lambda)$
with $\delta = \eps/2$ and $\lambda = 1/(2(1 - \eps))$,
noting that $\lambda = 1/(2(1 - \eps)) > 1/(2r_2) = 1/(2(1 - \eps/2))$.

For a Gaussian weight function, Example 4 in \cite{KSWW10}
(with $\nu = 1$ and $\alpha = 1/\mu$) gives
\[
C_2
 = \dfrac{\sqrt{2\pi}}{\pi^{2-2 \mu}(1- \mu) \mu}
 \quad \text{and} \quad r_2=1- \mu,
\]
In this case, the result \eqref{eq:qmc-rmse} with \eqref{eq:qmc-err-const-Gauss}
follows by letting
$\varrho(\eps) = 2C_2^\lambda \zeta(2r_2 \lambda)$ and
$\lambda = 1/(2(1 - \eps))$ for $\eps \in (\mu, 1/2]$,
noting that $\lambda = 1/(2(1 - \eps)) > 1/(2r_2) = 1/(2(1 - \mu))$
since $\eps > \mu$.
\end{proof}

To ensure that the preintegrated functions $g_\mathrm{cdf}$ and $g_\mathrm{pdf}$
as in \eqref{eq:g-xi} belong to $\calH^{\bsone}_{2s}$, and hence achieve the QMC error bound \eqref{eq:qmc-rmse}, we require that Assumptions~\ref{asm:phi} and \ref{asm:phi-h}
hold for $\bsnu = \bsone$, which will be verified in Section~\ref{sec:cond_preint}.

\section{PDE with random input}
\label{sec:pde}

In this section, we introduce the necessary background material on PDE
with random (specifically, i.i.d.\ normally distributed) inputs. First, we
start with the following notation. Let $V \coloneqq H_0^1(D)$ denote the
\emph{first-order Sobolev space} with vanishing boundary trace,  equipped
with the norm
$\|v\|_V \coloneqq \|\nabla v\|_{L^2(D)}$,
and denote the dual space by $V^* \coloneqq H^{-1}(D)$. We will also use
the notation $\langle \cdot, \cdot \rangle$ to denote the $L^2(D)$ inner
product, which can be extended continuously to a duality paring on $V^*
\times V$. Similarly, $W^{k, p}(D)$ denotes the usual Sobolev space of
smoothness $k \in \N$ and integrability $1 \leq p \leq \infty$, with the
Hilbert spaces ($p = 2$) denoted by $H^k(D) \coloneqq W^{k, 2}(D)$.

Next, we make the following assumptions on the physical domain $D$, the
coefficient $a$ and the source term $\ell$, which will ensure that the
variational form of PDE \eqref{eq:pde_0} is well posed.

\begin{assumption} \label{asm:pde}
Let $d\in \{1,2,3\}$. For the PDE problem \eqref{eq:pde_0} we assume
that
\begin{enumerate}[label=\textnormal{(\alph*)}] 
\item \label{asm:D} $D \subset \R^d$ is a bounded convex domain
with a $C^2$ boundary for $d \geq 2$, or, for $d = 2$,
$D \subset \R^2$ is a strictly convex polygon;

\item \label{asm:a} for $\bsz\in \bbR^s$, $a(\cdot, \bsz) \in W^{1,
    \infty}(D)$ is of the form \eqref{eq:axz} with 
    $a_j \in W^{1, \infty}(D)$ and $z_j \sim \rmN(0, 1)$ for $j = 1,
    2, \ldots, s$;
\item \label{asm:ell}
    for $\bsw\in \bbR^{s+1}$, $\ell(\cdot, \bsw)
    \in V^*$ is of the form \eqref{eq:lxw} with $\ellbar$, $\ell_i \in
    V^*$ and $w_i \sim \rmN(0, 1)$ for $i = 0, 1, 2, \ldots, s$; and
\item \label{asm:pos} $\ell_0 \in L^\infty(D)$ with $\ell_{0,\inf}
    \coloneqq \inf_{\bsx \in D} \ell_0(\bsx) > 0$.
\end{enumerate}
\end{assumption}
Here $W^{1, \infty}(D)$, the Sobolev space of functions with essentially bounded weak derivatives
on $D$, is equipped with the norm
$\|v\|_{W^{1,\infty}(D)} \coloneqq \max\{\|v\|_{L^\infty(D)},\|\nabla v\|_{L^\infty(D)}\}$.

Assumption~\ref{asm:pde}\ref{asm:a} implies that for $\bsz \in
\R^s$ there exists $\amin(\bsz) < \amax(\bsz)$ such that
\begin{align}\label{eq:a_bound}
    \amin(\bsz) \,\leq\, a(\bsx,\bsz) \,\leq\, \amax(\bsz)\quad
    \text{for all } \bsx \in D.
\end{align}
Furthermore, $\amax, 1/\amin \in L_{\bsrho}^p(\R^s)$ for all $1 \leq
p < \infty$, where $L_{\bsrho}^p(\R^s)$ denotes the space of
$p$-integrable functions on $\R^s$ with respect to the product normal
density $\bsrho$.

The smoothness of $u(\cdot, \bsy)$
will be governed by further regularity of $\ellbar, \ \ell_i$
beyond $V^*$. The extra condition on $\ell_0$ in
Assumption~\ref{asm:pde}\ref{asm:pos} is required to ensure that the
monotonicity condition for preintegration
(cf.~Assumption~\ref{asm:phi}\ref{itm:mono}) is satisfied, see
Section~\ref{sec:mono}.

\subsection{Parametric variational form}

In the usual way, the variational formulation is obtained by multiplying the PDE
\eqref{eq:pde_0} by a test function $v \in V$ then using Green's formula to integrate by parts with respect to $\bsx$, giving
\begin{align}\label{eq:weak_form}
		\int_D a(\bsx,\bsz)\,\nabla u(\bsx,\bsy)\cdot\nabla v(\bsx)\,\rd\bsx = \int_D \ell(\bsx,\bsw)\,v(\bsx)\,\rd\bsx,
	\end{align}
for some parameter
\[
  \bsy = (y_0,y_1,\ldots,y_s,y_{s+1},\ldots,y_{2s})
  = (w_0,w_1,\ldots,w_s,z_1,\ldots,z_s) = (\bsw,\bsz) \in \bbR^{2s+1}.
\]%
For $\bsz \in \bbR^s$,
define also the parametric bilinear form $\calA(\bsz; \cdot, \cdot) :V\times V\to
\R$ by
\begin{align}\label{eq:bilin}
\calA(\bsz; v,v')
\,\coloneqq\,
\int_D a(\bsx,\bsz)\,\nabla v(\bsx)\cdot\nabla v'(\bsx)\,\rd\bsx \quad \text{ for } v, v'\in V.
\end{align}
For any particular $\bsy\in\R^{2s+1}$,
the variational formulation of \eqref{eq:pde_0} is:
find $u(\cdot, \bsy) \in V$ such that
\begin{align}\label{eq:varpde}
\calA(\bsz; u(\cdot,\bsy),v) = \langle \ell(\cdot,\bsw),v \rangle \quad \text{for all } v\in V.
\end{align}

The condition \eqref{eq:a_bound} on the coefficient $a(\cdot, \bsz)$
implies that the bilinear form $\calA({\bsz}; \cdot, \cdot)$ is coercive and continuous, i.e., for any given
$\bsz \in \R^s$ and all $v,v'\in V$
\begin{align*}
	\calA(\bsz;v,v) \,\geq\, \amin(\bsz)\|v\|^2_V
	\qquad \text{and} \qquad
	|\calA(\bsz;v,v')| \,\leq\, \amax(\bsz)\|v\|_V \|v'\|_V.
\end{align*}
The Lax--Milgram Lemma then implies that
the variational problem \eqref{eq:varpde} admits a unique solution $u(\cdot, \bsy) \in V$,
with the following upper  bound
\begin{equation}
\label{eq:LaxMil}
\|u(\cdot, \bsy)\|_V \,\leq\,
\frac{\|\ell(\cdot, \bsw)\|_{V^*}}{\amin(\bsz)}
\,\leq\, \frac{1}{\amin(\bsz)}\bigg( \|\ellbar\|_{V^*} + \sum_{i = 0}^{s} |w_i| \, \|\ell_i\|_{V^*}\bigg),
\end{equation}
where in the second inequality we have used the specific form
\eqref{eq:lxw} of $\ell(\cdot, \bsw)$ and the triangle inequality.
Similarly, for $\calG \in V^*$, the quantity of interest
\eqref{eq:qoi} is bounded by
\begin{equation}
\label{eq:LaxMil-G}
|\calG(u(\cdot, \bsy))|
\,\leq\,
\|\calG\|_{V^*} \|u(\cdot, \bsy)\|_V
\,\leq\,
\frac{\|\calG\|_{V^*}\|\ell(\cdot, \bsw)\|_{V^*}}{\amin(\bsz)}\,.
\end{equation}

Due to linear structure of the source term \eqref{eq:lxw}, the solution
$u(\cdot, \bsy) \in V$ of \eqref{eq:varpde} can be written as a sum of $s + 1$ solutions,
\begin{equation}
\label{eq:u-sum}
  u(\bsx,\bsy) \,=\, \ubar(\bsx,\bsz) + \sum_{i=0}^s w_i\, u_i(\bsx,\bsz),
\end{equation}
where $\ubar(\cdot,\bsz) \in V$ and $u_i(\cdot,\bsz) \in V$ denote the
solutions to \eqref{eq:varpde} with source terms $\ellbar \in V^*$ and
$\ell_i \in V^*$, respectively, and they are independent of $\bsw$. Again the Lax--Milgram Lemma ensures that the solutions $\ubar(\cdot,
\bsz) \in V$ and $u_i(\cdot, \bsz) \in V$ are unique. In addition, both of
the a priori bounds \eqref{eq:LaxMil} and \eqref{eq:LaxMil-G} also hold
for $\ubar(\cdot, \bsz), \ u_i(\cdot, \bsz)$ but with $\ellbar,\ \ell_i$
on the right, respectively. Similarly, the quantity of interest
\eqref{eq:qoi} can be expressed as
\begin{equation} \label{eq:QoI2}
\calG(u(\cdot, \bsy)) \,=\, \calG(\ubar(\cdot, \bsz)) + \sum_{i = 0}^s w_i\ \calG(u_i(\cdot, \bsz)).
\end{equation}

\subsection{Parametric regularity}

To apply the preintegration theory from Section~\ref{sec:preint-theory},
as well as the QMC theory from Section~\ref{sec:qmc}, to this PDE problem,
we require that the solution $u$ is sufficiently regular with respect to the
 parameter $\bsy$.
Parametric regularity results for an arbitrary right-hand side are given in, e.g., \cite{GrKNSSS15},
which can easily be adapted to our setting as in the following theorem.

For readability, we adopt the notation
\begin{equation}
\label{eq:bj}
    b_j \,\coloneqq\, \|a_j\|_{L^{\infty}(D)}, 
    \quad     \cbar \,\coloneqq\, \|\ellbar\|_{V^*}
        \quad \text{and}
    \quad c_i\,\coloneqq\, \|\ell_i\|_{V^*}\,.
\end{equation}

We also split the order $\bsnu \in  \N_0^{2s + 1}$ for the derivative with
respect to $\bsy = (\bsw, \bsz)$ by writing $\bsnu = (\bsnu_\bsw,
\bsnu_\bsz)$, where $\bsnu_\bsw = (\nu_0, \nu_1, \ldots, \nu_s)$
corresponds to the derivatives with respect to $\bsw = (y_0, y_1,
\ldots, y_s)$ and $\bsnu_\bsz = (\nu_{s + 1}, \nu_{s + 2}, \ldots,
\nu_{2s})$ corresponds to the derivatives with respect to $\bsz=(y_{s
+ 1}, y_{s + 2}, \ldots, y_{2s})$.

\begin{theorem}\label{thm:deriv_bnd}
Suppose that Assumption~\ref{asm:pde} holds. Then, for $\bsy = (\bsw,
\bsz) \in \R^{2s + 1}$, the order $\bsnu = (\bsnu_\bsw, \bsnu_\bsz) \in
\N_0^{2s + 1}$ derivative of the solution $u(\cdot, \bsy) \in V$ of
\eqref{eq:varpde} with respect to $\bsy$ satisfies
\begin{align} \label{eq:deriv_bnd}
  \|\partial^\bsnu  u(\cdot,\bsy)\|_{V}  
  \le
  \begin{cases}
  \displaystyle
  \frac{|\bsnu_\bsz|!}{(\ln 2)^{|\bsnu_\bsz|}}
  \bigg(\prod_{j=1}^s b_j^{\nu_{s+j}}\bigg)
  \frac{\overline{c} + \sum_{i=0}^s c_i\,|w_i|}{\amin(\bsz)}
  & \mbox{if } \bsnu_\bsw = \bszero,
  \\
  \displaystyle
  \frac{|\bsnu_\bsz|!}{(\ln 2)^{|\bsnu_\bsz|}}
  \bigg(\prod_{j = 1}^s b_j^{\nu_{s+j}}\bigg)
  \frac{c_k}{\amin(\bsz)}
  &
  \text{if } \bsnu_{\bsw} = \bse_k \text{ for } k \in \{0:s\},
\\
  0 & \mbox{otherwise}.
  \end{cases}
\end{align}
\end{theorem}
\begin{proof}

Differentiating the linear expansion of $u(\cdot, \bsy)$ in
\eqref{eq:u-sum} with respect to $\bsy$ gives
\begin{equation}
\label{eq:du-sum}
\partial^\bsnu  u(\bsx, \bsy)
\,=\, \partial^{\bsnu_\bsw}
\partial^{\bsnu_\bsz} \ubar(\bsx, \bsz)
+ \sum_{i = 0}^s (\partial^{\bsnu_\bsw} w_i)\,
  (\partial^{\bsnu_\bsz}  u_i(\bsx, \bsz)) \,
,
\end{equation}
where we used the fact that $u_i(\cdot, \bsz)$ is the solution to
\eqref{eq:varpde} with the (deterministic) source term $\ell_i$ and is
independent of $\bsw$. Hence, by \cite[Theorem~14]{GrKNSSS15}
\begin{align}
\label{eq:der-GrKNSSS}
\|\partial^{\bsnu_\bsz} u_i(\cdot, \bsz)\|_V
\,&\leq\,
\frac{|\bsnu_\bsz|!}{(\ln 2)^{|\bsnu_\bsz|}} \bigg(\prod_{j=1}^s \|a_j\|_{L^\infty(D)}^{\nu_{s+j}}\bigg)
\frac{\|\ell_i\|_{V^*}}{\amin(\bsz)}
 \,= \, \frac{|\bsnu_\bsz|!}{(\ln 2)^{|\bsnu_\bsz|}} \bigg(\prod_{j=1}^s b_j^{\nu_{s+j}}\bigg)
\frac{c_i}{\amin(\bsz)},
\end{align}
where we have used the definitions in \eqref{eq:bj}.
The same bounds hold for $\|\partial^{\bsnu_\bsz} \ubar(\cdot, \bsz)\|_V$
with $\|\ell_i\|_{V^*}$, $c_i$ replaced by $\|\ellbar\|_{V^*}$, $\cbar$.

If $\bsnu_\bsw = \bszero$ then in \eqref{eq:du-sum} each derivative
operator $\partial^{\bsnu_\bsw}$ disappears and by the triangle
inequality
\[
\|\partial^\bsnu  u(\cdot, \bsy)\|_V
\,\leq\, \|\partial^{\bsnu_\bsz}  \ubar(\cdot, \bsz)\|_V
+ \sum_{i = 0}^s |w_i|\,\|\partial^{\bsnu_\bsz} u_i(\cdot, \bsz)\|_V \,
.
\]
Substituting in the bounds \eqref{eq:der-GrKNSSS}
yields case 1 of \eqref{eq:deriv_bnd}.

If $\bsnu_\bsw = \bse_k$ for some $k\in \{0:s\}$, then
\eqref{eq:du-sum} becomes $\partial^\bsnu  u(\bsx, \bsy) =
\partial^{\bsnu_\bsz} u_k(\bsx, \bsz)$. Taking the $V$-norm and
substituting in \eqref{eq:der-GrKNSSS} for $i=k$ yields
case 2 of \eqref{eq:deriv_bnd}.

Finally, if $|\bsnu_\bsw| > 1$ then clearly
$\partial^{\bsnu_\bsw}\ubar(\cdot, \bsz) = 0$ and
$\partial^{\bsnu_\bsw} w_i = 0$ for all $i$.
Hence, by \eqref{eq:du-sum} we have $\partial^\bsnu  u(\cdot, \bsy) = 0$,
proving the final case in \eqref{eq:deriv_bnd}.
\end{proof}

\section{Density estimation for PDE with random input}
\label{sec:de4pde}

We now consider the special case
of approximating the cdf and pdf of the quantity of interest \eqref{eq:qoi}
corresponding to the solution of the PDE
\eqref{eq:pde_0}, then give the full details of our approximation method.
For the remainder of the paper we define
\begin{equation}
\label{eq:phi-qoi}
  \varphi(y_0,\bsy) \,\coloneqq\, \calG(u(\cdot,(y_0,\bsy))),
\end{equation}
where for notational convenience we write in this section (separating
out $y_0$)
\[
  \bsy = (y_1,\ldots,y_s,y_{s+1},\ldots,y_{2s})
  = (w_1,\ldots,w_s,z_1,\ldots,z_s) = (\bsw,\bsz) \in \bbR^{2s}.
\]

Using the expansion of the quantity of interest \eqref{eq:QoI2}, we can also
write $\varphi$ as a sum
\begin{equation}
\label{eq:phi-sum}
\varphi(y_0, \bsy) \,=\,
\varphi(\bsw, \bsz) \,=\, \phibar(\bsz) +
\sum_{i = 0}^s w_i \, \varphi_i(\bsz),
\end{equation}
where $\phibar(\bsz) \coloneqq \calG(\ubar(\cdot, \bsz))$
and $ \varphi_i(\bsz) \coloneqq \calG(u_i(\cdot, \bsz))$ for $ i = 0, 1, \ldots, s$.
It then follows that the implicit function $\xi$ defining the point of discontinuity as in
\eqref{eq:xi-def} satisfies
\begin{align*}
 \varphi(\xi(t, \bsy), \bsy) \,=\,
 \overline\varphi(\bsz) + \xi(t,\bsy) \,\varphi_0(\bsz) + \sum_{i=1}^s w_i\,\varphi_i(\bsz) = t,
\end{align*}
which can be rearranged to give the explicit expression
\begin{align} \label{eq:xi-explicit}
 \xi(t,\bsy) \,=\, \frac{t-\overline\varphi(\bsz)-\sum_{i=1}^s w_i\,\varphi_i(\bsz)}{\varphi_0(\bsz)} .
\end{align}
Of course, for $\xi$ to be well defined, we require $\varphi_0(\bsz) \ne
0$ for all $\bsz$, which follows from the Strong Maximum Principle since
we assumed in Assumption~\ref{asm:pde}\ref{asm:pos} that
$\ell_0(\bsx)> 0$ for all $\bsx\in D$. Further details on this
condition, along with an explicit lower bound on $\varphi_0(\bsz)$
are given in Section~\ref{sec:mono}.

Since $\phibar(\bsz)$ and $\varphi_i(\bsz)$ for $i = 0, 1, \ldots, s$
are independent of $\bsy_{1:s} = \bsw$, by the representation \eqref{eq:phi-sum}
the derivatives of $\varphi$ with respect to $y_i = w_i$ simplify to
\[
   \partial^{i} \varphi(y_0,\bsy) \,=\, \varphi_i(\bsz) \qquad \mbox{for all } 0, 1, \ldots, s.
\]
In particular, $ \partial^{0} \varphi(y_0,\bsy) = \varphi_0(\bsz)$
and the preintegrated function $g_\mathrm{pdf}$ in \eqref{eq:g-xi} becomes
\begin{equation}
\label{eq:g_pdf}
g_\mathrm{pdf}(\bsy) \,=\, \frac{\rho(\xi(t, \bsy))}{\varphi_0(\bsz)},
\end{equation}
where now $\xi$ is given by \eqref{eq:xi-explicit}.

To approximate the cdf of $X$ using the integral formulation
\eqref{eq:cdf-y}, preintegration is performed with respect to $y_0$ as in
\eqref{eq:preint-cdfpdf}, and then a QMC rule as in \eqref{eq:qmc} is
applied to approximate the remaining $2s$-dimensional integral. The pdf
\eqref{eq:pdf-y} is approximated in the same way. In this way, the
approximations of the cdf and pdf at $t \in [t_0, t_1]$ using
preintegration followed by an $N$-point QMC rule are denoted by
\begin{align}
\label{eq:F_N}
F(t) \,&\approx\, F_N(t)
\,\coloneqq\, Q_{2s, N} (g_\mathrm{cdf})
\,=\, Q_{2s, N} \big(\Phi(\xi(t, \cdot))\big),
\\
\label{eq:f_N}
f(t) \,&\approx\; f_N(t)
\,\coloneqq\, Q_{2s, N} (g_\mathrm{pdf})
\,=\, Q_{2s, N}
\bigg( \frac{\rho(\xi(t, \cdot))}{\varphi_0}\bigg),
\end{align}
where we have simplified the preintegrated functions $g_\mathrm{cdf}$ and
$g_\mathrm{pdf}$ as in \eqref{eq:g-xi} and \eqref{eq:g_pdf}, and the
point of discontinuity is given explicitly by \eqref{eq:xi-explicit}.

In practice, to compute the approximations $F_N(t)$ in
\eqref{eq:F_N} and $f_N(t)$ in \eqref{eq:f_N}, we must evaluate
$\xi(t, \bsy)$, which in turn requires solving the PDE numerically to
obtain discrete approximations of $\phibar(\bsz)$ and $\varphi_i(\bsz)$,
e.g., using finite element methods. To ensure that the discrete
approximation of $\xi(t, \bsy)$ in \eqref{eq:xi-explicit} is well defined,
we require that the discrete approximation of $\varphi_0(\bsz)$ is
non-zero for all $\bsz$. Verifying this monotonicity condition for the
discrete approximation of $\varphi_0(\bsz)$, as well as analysing the
extra discretisation error, adds a significant amount of {technical
analysis}, and so will be handled in a separate paper \cite{Gil24}.

\section{Conditions required for preintegration theory}
\label{sec:cond_preint}

To apply the theory for preintegration to our quantity of interest \eqref{eq:phi-qoi},
we must verify that the conditions outlined in
Assumptions~\ref{asm:phi} and~\ref{asm:phi-h} are satisfied.
This will be done in the subsections to follow. 
To apply the theory for first-order QMC rules from Section~\ref{sec:qmc}, 
we only require these assumptions to 
hold for $\bsnu = \bsone$, however, we verify the conditions for general 
$\bsnu \in \N_0^{2s}$.
Similarly, since we only consider the cdf and pdf we require $q \in \{0, 1\}$ in 
Assumption~\ref{asm:phi-h}, but we give results for general $q \in \N_0$.

First, note that the Gaussian density satisfies $\rho \in C^\infty(\R)$.
Also, since $\varphi_i$ is continuous for all $i\in\{1:s\}$ and
$\varphi_0$ is strictly positive (as we will soon verify), we also have
that $U_t = \R^{2s}$ in \eqref{eq:Ut}, implying that the condition
\eqref{eq:h-lim} is no longer necessary because there is no boundary $\partial U_t$ over which the continuity  of the functions $h_{q,\boldsymbol{\eta}}(t,\boldsymbol{y})$ needs to be ensured.

\subsection{Verifying the monotonicity condition}
\label{sec:mono}

Assumption~\ref{asm:phi}\ref{itm:mono} requires our quantity of
interest to be monotone with respect to the preintegration variable
$y_0$. This monotone property simplifies the
practical implementation of preintegration by ensuring that
the point of discontinuity \eqref{eq:xi-def} is unique.
It was shown in \cite{GKS22b} to be also necessary for the smoothing theory, since if it fails to hold then there may still be singularities after preintegration.

From \eqref{eq:QoI2}, we note that
\[
  \partial^{0}\varphi(y_0,\bsy) = \varphi_0(\bsz) = \calG(u_0(\cdot,\bsy)).
\]
As $\calG$ is a linear functional, we can guarantee monotonicity if $u_0$
is strictly positive or negative for all $\bsx\in D$ and $\bsy\in
\R^{2s+1}$, and correspondingly $\calG$ is also strictly positive or
negative. Since $u_0$ is the solution to \eqref{eq:pde_0} with $\ell_0$
as the source term, by the Weak Maximum Principle (see, e.g.,
\cite{GilbTrud}) the positivity of the solution is guaranteed if
$\ell_0(\bsx)>0$ for all $\bsx\in D$, as assumed in
Assumption~\ref{asm:pde}\ref{asm:pos}.

In particular, the lower bound on the PDE solution in
Lemma~\ref{lem:u0_lower} below can be used to obtain a lower bound on
$\varphi_0$, which implies the necessary monotonicity condition and will
also be useful later when deriving the bounds $B_{q, \bseta}$ from
Assumption~\ref{asm:phi-h}.

First, we show that $u_0(\bsx, \bszero)$, i.e., the solution to
\eqref{eq:varpde} with $\bsz = \bszero$ and right-hand side $\ell_0$, has
bounded first derivatives with respect to $\bsx$.

\begin{lemma}
\label{lem:u0-reg} Suppose that Assumption~\ref{asm:pde} holds. Then
the solution to \eqref{eq:varpde} corresponding to the right-hand
side $\ell_0$ with $\bsz = \bszero$ satisfies $u_0(\cdot, \bszero) \in
W^{1, \infty}(D)$.
\end{lemma}

\begin{proof}
Since the smoothness assumptions on the domain are different in each
dimension, we first prove the result for $d = 1$ or $d =2,3$ with a $C^2$ boundary,
then for $D$ a convex polygon when $d = 2$ separately.

Under Assumption~\ref{asm:pde}\ref{asm:D}, $D$ is a bounded, open interval
when $d=1$, and $D$ is convex with a $C^2$ boundary when $d =$ 2,
3. Hence, since $\ell_0 \in L^\infty(D)$ and $a(\cdot, \bszero) =
a_0 \in W^{1, \infty}(D)$, it follows from
\cite[Theorem~9.15]{GilbTrud} that $u_0(\cdot,\bszero) \in W^{2,
p}(D)$ for all $1 < p < \infty$. It then follows by the Sobolev Embedding
Theorem \cite[Corollary~7.11]{GilbTrud} that $u_0(\cdot,\bszero) \in
C^1(\overline{D}) \subset W^{1, \infty}(D)$, as required.

For $d = 2$ and $D$ a convex polygon, following \cite{Gris11} we can
write $u_0(\bsx, \bszero)$ as the sum of a $W^{2, p}(D)$ function 
(for $p > 2$ specified below) and an
expansion of the singularities at each of the corners. Let $K$ denote the
number of sides in the strictly convex polygon $D$ and let $\omega_k <
\pi$, for $k = 1, 2, \ldots, K$, denote a labelling of the angles. Choose
$p > 2$ such that $p^{-1} + q^{-1} = 1$ and $q > 1$ satisfies $2\omega_k
/(q\pi) \not \in \Z$ for all $k = 1, 2, \ldots, K$. It follows from
\cite[Theorem~4.4.3.7]{Gris11} that there exists $\widetilde{u} \in W^{2,
p}(D)$ and $C_{k, m} \in \R$ such that
\begin{equation}
\label{eq:u0-singular}
 u_0(\bsx, \bszero) \,=\, \widetilde{u}(\bsx) +
\sum_{k = 1}^K\sum_{\substack{m \in \Z\\ -2/q < \lambda_{k, m} < 0 \\ \lambda_{k, m} \neq -1}}
C_{k, m}\, S_{k, m}(\bsx),
\end{equation}
where for $m \in \Z$, $k = 1, 2, \ldots, K$, we define $\lambda_{k, m} \coloneqq m \pi / \omega_k$
and the $m$th singularity at the $k$th corner is given by
\[
S_{k, m}(\bsx) = S_{k, m}(r_k\, e^{i \theta_k})
\,\coloneqq\, \frac{r_k^{-\lambda_{k, m}}}{\sqrt{\omega_k}\, \lambda_{k, m}}
\cos (\lambda_{k, m}\,\theta_k + \tfrac{\pi}{2})\, \eta_k(r_k\, e^{i\theta_k})\,.
\]
We have written the input $\bsx$ in polar coordinates $(r_k, \theta_k)$
centred at the $k$th corner and $\eta_k$ is a smooth cutoff function,
centred at the $k$th corner and with support that only intersects the
adjacent edges. Precise details can be found in \cite{Gris11}.

The Sobolev Emdedding Theorem \cite[Corollary~7.11]{GilbTrud} again implies that
$\widetilde{u} \in W^{2, p}(D) \subset W^{1, \infty}(D)$, because $p > d = 2$.

Since $\omega_k < \pi$ and $q > 1$, for $k = 1, 2, \ldots, K$ we have
$\lambda_{k, -1} < -1$, and for $m \leq -2$ we have $\lambda_{k, m} <
-m \le -2 < -2/q$. Therefore, in \eqref{eq:u0-singular} it is only
possible for $\lambda_{k, -1}$ to be in $(-2/q, 0)\setminus \{-1\}$.
Similarly, if $\omega_k \leq \pi q/2$ then $\lambda_{k, -1} \leq -2/q$,
otherwise if $\omega_k > \pi q/2$ then $-2/q < \lambda_{k, -1} < -1$.
Hence, we can simplify \eqref{eq:u0-singular} to
\begin{equation}
\label{eq:u0-singular2}
u_0(\cdot, \bszero) \,=\, \widetilde{u}(\bsx) +
\sum_{\substack{k = 1\\ \omega_k > \pi q/2}}^K
C_{k, -1} \,S_{k, -1}(\bsx).
\end{equation}

In polar coordinates, the gradient of the singularity at the $k$th corner, $S_{k, -1}$, is
\[
 \nabla S_{k, -1}(r_k\,e^{i\theta_k}) \,=\,
 \bigg(
 \pd{}{}{r_k} S_{k, -1}(r_k\, e^{-\theta_k}),
 \frac{1}{r_k} \pd{}{}{\theta_k} S_{k, -1}(r_k\, e^{i\theta _k})
 \bigg)^\top,
\]
where
\begin{align*}
\pd{}{}{r_k} S_{k, -1}(r_k\, e^{-\theta_k})
\,=\, &-\frac{r_k^{-\lambda_{k, -1} - 1}}{\sqrt{\omega_k}}
\cos(\lambda_{k, -1}\, \theta_k + \tfrac{\pi}{2})\, \eta_k(r_k\, e^{-\theta_k})
\\
&+ \frac{e^{i\theta_k}\, r_k^{-\lambda_{k, -1}}}{\sqrt{\omega_k}\, \lambda_{k, -1}}
\cos(\lambda_{k, -1}\, \theta_k + \tfrac{\pi}{2})\, \eta_k'(r_k\, e^{-\theta_k}),
\end{align*}
which is bounded since $\lambda_{k, -1} < -1$ implies that $-\lambda_{k, -1} - 1 > 0$
and $\eta_k$ is a smooth cutoff function. Similarly, we have
\begin{align*}
\frac{1}{r_k} \pd{}{}{\theta_k} S_{k, -1}(r_k\, e^{i\theta _k})
\,=\, &-\frac{r_k^{-\lambda_{k, m} - 1}}{\sqrt{\omega_k}}
\sin(\lambda_{k, -1}\, \theta_k + \tfrac{\pi}{2})\, \eta_k(r_k\, e^{i\theta_k})
\\
&+\frac{r_k^{-\lambda_{k, -1}}\, ie^{i\theta_k} }{\sqrt{\omega_k}\, \lambda_{k, m}}
\cos(\lambda_{k, -1}\, \theta_k + \tfrac{\pi}{2})\, \eta_k' (r_k\, e^{i\theta_k})\,,
\end{align*}
which is again bounded since $\lambda_{k, -1} < -1$ and $\eta_k$ is
smooth. Hence,  $S_{k, -1} \in W^{1, \infty}(D)$ for all $k = 1, 2,
\ldots, K$, and in turn from \eqref{eq:u0-singular2} it follows that
$u_0(\cdot, \bszero) \in W^{1, \infty}(D)$.
\end{proof}

\begin{remark}
For $d = 3$, it may be possible to use similar arguments as in $d = 2$
to obtain the required regularity under the relaxed assumption of a
strictly convex polyhedral domain, instead of a $C^2$ domain as in
Assumption~\ref{asm:pde}\ref{asm:D}. However, such analysis is beyond the
scope of this paper.
\end{remark}

The extra regularity of $u_0(\cdot, \bszero)$ allows us to prove the
following lower bound using the Weak Maximum Principle (e.g.,
\cite[Theorem~8.1]{GilbTrud}).

\begin{lemma}
\label{lem:u0_lower} Suppose that Assumption~\ref{asm:pde} holds, thus
$\ell_{0, \mathrm{inf}} \coloneqq \inf_{\bsx \in D} \ell_0(\bsx) > 0$. 
Then the solution $u_0(\cdot, \bsz) \in V$ to \eqref{eq:varpde} with
right-hand side $\ell_0$ satisfies
\[
u_0(\bsx, \bsz) \,\geq\, \frac{u_0(\bsx, \bszero)}{K_0(\bsz)}
\quad \text{for almost all } \bsx \in D \mbox{ and all } \bsz \in \R^{s},
\]
where
\begin{align} \label{eq:K0}
K_0(\bsz) \,\coloneqq\, \amax(\bsz)\bigg( 1
+ \frac{\|u_0(\cdot, \bszero)\|_{W^{1, \infty}(D)}}{\ell_{0, \mathrm{inf}}}
\sum_{j = 1}^s |z_j|\, \|a_j\|_{W^{1, \infty}(D)}
\bigg).
\end{align}
\end{lemma}

\begin{proof}
We prove the result by applying the Weak Maximum Principle to the
function
\begin{equation}
\label{eq:uhat}
\uhat(\bsx, \bsz) \,\coloneqq\, \frac{u_0(\bsx, \bszero)}{K_0(\bsz)} - u_0(\bsx, \bsz)
\;\in\; V\,.
\end{equation}

First we show that $\calA(\bsz; \uhat(\cdot, \bsz), v) \leq 0$ for all $v
\in V$ such that $v \geq 0$. By considering the bilinear form
\eqref{eq:bilin} (for parameter value $\bsz$) with argument $u_0(\cdot,
\bszero)$ (i.e., the solution with parameter value $\bszero$),
integrating by parts, and applying the product rule, we obtain
\begin{align*}
&\calA(\bsz; u_0(\cdot, \bszero), v)
\,=\,
\int_D a(\bsx, \bsz)\, \nabla u_0(\bsx, \bszero) \cdot \nabla v(\bsx) \, \rd \bsx
\\
&\,=\, \int_D -\nabla \cdot \big(a(\bsx, \bsz)\, \nabla u_0(\bsx, \bszero)\big) \, v(\bsx) \, \rd \bsx
\\
&\,=\, \int_D \bigg(- a(\bsx, \bsz) \nabla \cdot \big[ a(\bsx,\bszero) \nabla u_0(\bsx, \bszero)\big]
 - \nabla a(\bsx, \bsz) \cdot \nabla u_0(\bsx, \bszero) \bigg)v(\bsx) \, \rd \bsx
\\
&\,=\, \int_D \bigg( a(\bsx,\bsz)\,\ell_0(\bsx) -
      a(\bsx, \bsz) \sum_{j = 1}^s z_j \nabla a_j(\bsx) \cdot \nabla u_0(\bsx,
      \bszero) \bigg)v(\bsx) \, \rd \bsx\,,
\end{align*}
where we inserted $a(\bsx,\bszero)=1$ and used the fact that
$u_0(\bsx, \bszero)$ satisfies the strong form \eqref{eq:pde_0}, with
$\bsz = \bszero$ and right-hand side $\ell_0$, almost everywhere in $D$.

Taking the absolute value and using the fact that $v(\bsx) \geq 0$
and $\ell_0(\bsx)> 0$ gives the upper bound
\begin{align} \label{eq:Au0<Auz}
&\calA(\bsz; u_0(\cdot, \bszero), v) \,\leq\,
|\calA(\bsz; u_0(\cdot, \bszero), v) |
\nonumber\\
&\leq\,
\int_D \bigg(a(\bsx,\bsz)\,
\ell_0(\bsx)
+
a(\bsx, \bsz) \frac{\ell_0(\bsx)}{\ell_0(\bsx)}
\sum_{j = 1}^s |z_j|\, |\nabla a_j(\bsx)|\, |\nabla u_0(\bsx, \bszero) |
\bigg) v(\bsx) \, \rd \bsx
\nonumber\\
&\leq\,
\amax(\bsz)\bigg( 1
+ \frac{\|u_0(\cdot, \bszero)\|_{W^{1, \infty}(D)}}{\ell_{0,\inf}}
\sum_{j = 1}^s |z_j|\, \|a_j\|_{W^{1, \infty}(D)} \bigg) \langle \ell_0, v\rangle
\nonumber\\
&= K_0(\bsz)\, \calA(\bsz; u_0(\cdot, \bsz), v),
\end{align}
where by Lemma~\ref{lem:u0-reg} we have $u_0(\cdot, \bszero) \in W^{1, \infty}(D)$,
and we used the definition \eqref{eq:K0} of $K_0(\bsz)$, as well
as the fact that $u_0(\cdot, \bsz)$ satisfies \eqref{eq:varpde} with the
right hand side $\ell_0$.

Rearranging \eqref{eq:Au0<Auz}, it follows that the function $\uhat(\cdot, \bsz)$ in \eqref{eq:uhat} satisfies
\[
\calA(\bsz, \uhat(\cdot, \bsz), v) \,\leq\, 0
\quad \text{for all } v \in V, v \geq 0.
\]
The Weak Maximum Principle, e.g., \cite[Theorem~8.1]{GilbTrud},
then implies that, for all $\bsz$
\[
\sup_{\bsx \in D} \uhat(\bsx, \bsz) \,\leq\, \sup_{\bsx \in \partial D} \uhat(\bsx, \bsz) \,=\, 0\,,
\]
since $\uhat(\cdot, \bsz) \in V$. This implies $\uhat(\bsx, \bsz) \leq 0$
for almost all $\bsx \in D$, which can be rearranged to give the
required result.
\end{proof}

Finally,
in the case where the quantity of interest \eqref{eq:qoi} is given by a positive linear
functional, i.e., $v \geq 0$ implies $\calG(v) \geq 0$,
this lower bound can be used to prove a lower bound on the term $\varphi_0$.

\begin{corollary}
\label{cor:lower}
Let $\calG \in V^*$ be a positive linear functional, then
\begin{equation*}
\varphi_0(\bsz) \,=\, \calG(u_0(\cdot, \bsz)) \,\geq\,
\frac{\calG(u_0(\cdot, \bszero))}{K_0(\bsz)}
\,=\, \frac{\varphi_0(\bszero)}{K_0(\bsz)}.
\end{equation*}
\end{corollary}
\begin{proof}
The proof of Lemma~\ref{lem:u0_lower} implies that $-\uhat(\cdot, \bsz) \geq 0$, cf. \eqref{eq:uhat},
then since $\calG$ is positive and linear
\[
0 \,\leq\, \calG(-\uhat(\cdot, \bsz)) \,=\, \calG(u_0(\cdot, \bsz))
- \frac{\calG(u_0(\cdot, \bszero))}{K_0(\bsz)}\,,
\]
which can be rearranged to give the required result.
\end{proof}

\begin{remark}
\label{rem:extensions}
Given the importance of the monotonicity condition for preintegration theory
and our method, one may wonder if it holds for different settings
for the source term, e.g., a different form of $\ell(\bsx, \omega)$ or a deterministic source term.
If the source term is a nonlinear function of the sum in \eqref{eq:lxw}, 
e.g., $\ell(\bsx, \omega) = L(\sum_{j = 0}^s y_j \ell_j(\bsx))$ for a sufficiently smooth function $L: \R \to \R$,
then the derivative of the solution with respect to $y_0$ again satisfies the 
variational equation \eqref{eq:varpde} but with a different source term,
\[
\calA \bigg( \bsz; \frac{\partial}{\partial y_0} u(\cdot, \bsy), v \bigg) = \bigg\langle \ell_0 L'\bigg(\sum_{j = 0}^s y_j \ell_j\bigg), v \bigg\rangle
\quad \text{for all } v \in V.
\]
If this nonlinear function $L$ is monotone, then similar arguments
as above can be used to show that the quantity of interest is also 
monotone.

Alternatively, for a deterministic source term, $\ell(\bsx, \omega) = \ell(\bsx)$,
clearly, the preintegration variable must be selected from the diffusion coefficient.
In this case, the monotonicity condition \eqref{eq:mono} must also be verified for the specific choice of coefficient in order for the preintegration theory to apply.
Consider a simple example of the PDE \eqref{eq:pde_0} with deterministic source term $\ell(\bsx)$ and a simple stochastic coefficient $a(\bsx, \bsy) = e^{y_0 + a_1(\bsx) y_1}$, i.e., $-\nabla\cdot(e^{y_0+a_1(\bsx)y_1}\nabla u(\bsx,\bsy))=\ell(\bsx)$. Since $e^{y_0}$ is independent of $\bsx$, it possible to rearrange the PDE as $-\nabla~\cdot~(e^{a_1(\bsx)y_1}\nabla u(\bsx,\bsy))= e^{-y_0}\ell(\bsx)$, which is in the same form as the problem this paper considers except with a lognormally distributed random source term. Hence, the monotonicity condition holds in this case.
With a less trivial diffusion term, e.g., $-\nabla\cdot(e^{a_0(\bsx)y_0+a_1(\bsx)y_1}\nabla u(\bsx,\bsy))=\ell(\bsx)$, determining monotonicity becomes harder.

For both cases above, the point of discontinuity $\xi(t, \bsy)$ satisfying 
\eqref{eq:xi-def} may not have a closed form, in which case to perform
preintegration in practice $\xi(t, \bsy)$ must be found numerically, which
has previously been considered in \cite{BayB-HTemp22,GKLS18}.
\end{remark}

\subsection{Verifying the smoothness of the quantity of interest}

In this section, we verify the remaining smoothness conditions from
Assumption~\ref{asm:phi} on the quantity of interest \eqref{eq:phi-qoi}
and the density $\rho$.

First, since $\varphi$ is linear with respect to $y_0 = w_0$ in
\eqref{eq:phi-sum}, it follows that $\varphi(y_0, \bsy) \to \infty$ as
$y_0 \to \infty$. Thus, Assumption~\ref{asm:phi}\ref{itm:inf} is
satisfied.

Next, we verify the ``Sobolev'' and ``classical'' smoothness conditions in
Assumption~\ref{asm:phi}\ref{itm:phi-smooth} separately. The bounds on the
derivatives of the solution in \eqref{eq:deriv_bnd} can be used to show
that $\varphi \in \calH^{\widetilde{\bsnu}}_{2s + 1}$ for any
$\widetilde{\bsnu}\in \N_0^{2s + 1}$, for both exponential
\eqref{eq:psi-exp} and Gaussian \eqref{eq:psi-Gauss} weight functions.
For the classical smoothness condition, it was shown in \cite[Proposition~3.8]{DunNguSchwZech23} that
$\ubar(\cdot, \bsz) \in V$ and each $u_i(\cdot, \bsz) \in V$ for $ i = 0,
1, \ldots, s$ are holomorphic in a complex neighbourhood of $\bsz \in
\R^s$. It follows that $\phibar(\bsz), \varphi_i(\bsz) \in \R$ are also
holomorphic in a complex neighbourhood of $\bsz \in \R^s$ and hence, the
derivatives of $\phibar, \varphi_i$ of all orders are continuous on
$\R^s$. Since $\varphi(\bsy)$ can be written as a sum as in
\eqref{eq:phi-sum}, it follows that $\varphi \in
C^{\widetilde{\bsnu}}(\R^{2s+1})$ for any $\widetilde{\bsnu} \in \N_0^{2s
+ 1}$. Hence, $\varphi$ satisfies
Assumption~\ref{asm:phi}\ref{itm:phi-smooth} for any $\bsnu \in
\N_0^{2s}$.

Finally, since $\rho$ is the Gaussian density, it is in $C^\infty(\R)$ and
so $\rho$ clearly satisfies $\rho\in C^{|\bsnu|}(\bbR)$
for any $\bsnu \in \N_0^{2s}$ as required.
  	
\subsection{Bounding the derivative terms}

Having verified Assumption~\ref{asm:phi} for our quantity of interest
\eqref{eq:phi-qoi}, we now  verify Assumption~\ref{asm:phi-h} by obtaining explicit constants $B_{q,\bseta}$ in
\eqref{eq:h-int} for all functions $h_{q,\bseta}(t,\bsy)$ of the form
\eqref{eq:h-form}. These functions represent typical functions that arise
from taking mixed derivatives of the preintegrated functions
\eqref{eq:g-xi}, and the constants $B_{q,\bseta}$ then provide bounds on
the $\calH_{2s}^{\bsnu}$-norm of the preintegrated functions for the cdf
and pdf, see \eqref{eq:norm-cdf} and \eqref{eq:norm-pdf}.

\begin{lemma}\label{lem:h_bound}
Let $\varphi$ be     the  quantity of interest \eqref{eq:phi-qoi} for the
PDE problem \eqref{eq:varpde}. Given $q\in\bbN_0$ 
and $\bseta\in\bbN_0^{2s}$, for $t\in
[t_0,t_1]$ and $\bsy = (\bsw,\bsz)\in \R^{2s}$, the functions
$h_{q,\bseta}(t,\bsy)$ of the form \eqref{eq:h-form} can be bounded by
\begin{align} \label{eq:h-bound}
 |h_{q,\bseta}(t,\bsy)|
 &\leq 
		\frac{1}{\sqrt{2\pi}}\frac{|\bseta|!\sqrt{(|\bseta| + q - 1)!}}{(\ln 2)^{|\bseta|}}
       \bigg(\prod_{j=1}^s b_j^{\eta_{j+s}} \bigg)
      \bigg(\prod_{i=1}^s c_i^{\eta_{i}} \bigg) 
\\ \nonumber
 &\quad\cdot \frac{1}{\varphi_0(\bsz)^{4|\bseta|+3q - 2}}
 \bigg(\frac{c_0\|\calG\|_{V^*}}{a_{\min}(\bsz)}
 \bigg[|t| + \frac{2\|\calG\|_{V^*}}{a_{\min}(\bsz)}\bigg(\bar{c} + c_0 + 1 + \sum_{i=1}^s c_i |w_i|\bigg)
 \bigg]\bigg)^{2|\bseta| + q - 1}.
\end{align}
\end{lemma}

\begin{proof}
Starting from \eqref{eq:h-form}, for parameters $(r,\bsalpha,\beta)$,
restricted such that the conditions in \eqref{eq:h-form} hold, we have
\begin{align*}
 |h_{q,\bseta}(t,\bsy)|
 = |h_{q,\bseta,(r,\bsalpha,\beta)}(t,\bsy)|
 &= |\rho^{(\beta)}(\xi(t,\bsy))|
 \frac{\prod_{\ell=1}^r |\partial^{\bsalpha_\ell}\varphi(\xi(t,\bsy),\bsy)|}{\varphi_0(\bsz)^{r+q}} 
\\
 &\le \|\rho^{(\beta)}\|_{L^\infty(\bbR)}\,\|\calG\|_{V^*}^r
 \frac{\prod_{\ell=1}^r \|\partial^{\bsalpha_\ell}u(\cdot,\xi(t,\bsy),\bsy)\|_V}{\varphi_0(\bsz)^{r+q}},
\end{align*}
where $\rho^{(\beta)}(x) = (-1)^\beta\,{\rm He}_\beta(x)\,\rho(x)$ with ${\rm
He}_\beta$ denoting the order $\beta$ (probabilist's) Hermite polynomial.
Every polynomial is dominated by the normal density $\rho$
and in particular, it follows from \cite{Ind61} that $\|\rho^{(\beta)}\|_{L^\infty(\R)} \leq  \sqrt{\beta!/(2\pi)}$.

Ultimately, we want an  upper bound which does not depend on $(r,\bsalpha,\beta)$. To this end,                      since $\beta \leq |\bseta| + q - 1$, we have 
\begin{equation}
\label{eq:h1}
 |h_{q,\bseta}(t,\bsy)|
\le \sqrt{\frac{(|\bseta| + q - 1)!}{2\pi}}\,\|\calG\|_{V^*}^r
 \frac{\prod_{\ell=1}^r \|\partial^{\bsalpha_\ell}u(\cdot,\xi(t,\bsy),\bsy)\|_V}{\varphi_0(\bsz)^{r+q}}.
\end{equation}

Next, we bound the derivative terms in the product.
Using Theorem \ref{thm:deriv_bnd}, we obtain
\begin{align*}
 &\|\partial^{\bsalpha_{\ell}} u(\cdot,\xi(t,\bsy),\bsy)\|_{V} \\
 &\le
 \begin{cases}
 \displaystyle
\frac{|(\bsalpha_\ell)_{s+1:2s}|!}{(\ln 2)^{|(\bsalpha_\ell)_{s+1:2s}|}}
			 \bigg(\prod_{j=1}^s b_j^{\alpha_{\ell,j+s}}\bigg)
 \frac{S(\bsy)}{a_{\min}(\bsz)}
 & \mbox{if } (\bsalpha_\ell)_{0:s} = \bszero,
 \\
 \displaystyle
 \frac{|(\bsalpha_\ell)_{s+1:2s}|!}{(\ln 2)^{|(\bsalpha_\ell)_{s+1:2s}|}}
			 \bigg(\prod_{j=1}^s b_j^{\alpha_{\ell,j+s}}\bigg)
			 \frac{c_k}{a_{\min}(\bsz)}
 &
\mbox{if } (\bsalpha_\ell)_{0:s} = \bse_k \mbox{ for some } k\in \{0:s\},
 \\
 0 & \text{otherwise},
 \end{cases}
\end{align*}
where we define $S(\bsy) \coloneqq \bar{c}+c_0|\xi(t,\bsy)|+\sum_{i=1}^s c_i|w_i|$.
These three cases can be combined into a single upper bound
\begin{align}
\label{eq:du-single}
 \|\partial^{\bsalpha_{\ell}} u(\cdot,\xi(t,\bsy),\bsy)\|_{V}
 &\le
\frac{|(\bsalpha_\ell)_{s+1:2s}|!}{(\ln 2)^{|(\bsalpha_\ell)_{s+1:2s}|}}
			 \bigg(\prod_{j=1}^s b_j^{\alpha_{\ell,j+s}}\bigg)
 \bigg(\prod_{i=1}^s c_i^{\alpha_{\ell,i}}\bigg)
 \frac{\max(S(\bsy),c_0,1)}{a_{\min}(\bsz)}.
\end{align}
Indeed, for the first case $(\bsalpha_\ell)_{0:s} = \bszero$ we have
$\prod_{i=1}^s c_{i}^{\alpha_{\ell,i}} = 1$ and the factor $S(\bsy)$ in the
numerator above $a_{\min}(\bsz)$. For the second case
$(\bsalpha_\ell)_{0:s} = \bse_k$, if $k = 0$ then $\prod_{i=1}^s
c_{i}^{\alpha_{\ell,i}} = 1$ and we have the factor $c_0$, but if
$k\ne 0$ then $\prod_{i=1}^s c_{i}^{\alpha_{\ell,i}}=c_k$ which
incorporates $c_k$ so we have the factor $1$. The bound overestimates
the $0$ value for all other cases of $\bsalpha_\ell$.

Substituting \eqref{eq:du-single} into \eqref{eq:h1} gives
\begin{align} \label{eq:hT1T2bd}
 |h_{q,\bseta}(t,\bsy)|
 \le 
 \sqrt{\frac{(|\bseta| + q - 1)!}{2\pi}}
 \,T_1\, T_2(\bsy),
\end{align}
where we define
\begin{align*}
 T_1 &\coloneqq
 \prod_{\ell=1}^r \bigg[\frac{|(\bsalpha_\ell)_{s+1:2s}|!}{(\ln 2)^{|(\bsalpha_\ell)_{s+1:2s}|}}
	\bigg(\prod_{j=1}^s b_j^{\alpha_{\ell,j+s}}\bigg)
	\bigg(\prod_{i=1}^s c_i^{\alpha_{\ell,i}}\bigg)\bigg], 
\\
 T_2(\bsy)
 &\coloneqq  \frac{1}{\varphi_0(\bsz)^q}
 \bigg(\frac{\max(S(\bsy),c_0,1)\,\|\calG\|_{V^*}}{\varphi_0(\bsz)\,a_{\min}(\bsz)}\bigg)^r.
\end{align*}
We can write
\begin{align} \label{eq:T1bnd}
 T_1
 &= \bigg(\prod_{\ell=1}^r |(\bsalpha_\ell)_{s+1:2s}|!\bigg)
    \bigg(\frac{1}{(\ln 2)^{\sum_{\ell=1}^r |(\bsalpha_\ell)_{s+1:2s}|}}\bigg)
    \bigg(\prod_{j=1}^s b_j^{\sum_{\ell=1}^r\alpha_{\ell,j+s}} \bigg)
	\bigg(\prod_{i=1}^s c_i^{\sum_{\ell=1}^r\alpha_{\ell,i}} \bigg) \notag \\
 &\le \frac{|\bseta|!}{(\ln 2)^{|\bseta|}}
    \bigg(\prod_{j=1}^s b_j^{\eta_{j+s}} \bigg)
    \bigg(\prod_{i=1}^s c_i^{\eta_{i}} \bigg),
\end{align}
where we used $\sum_{\ell=1}^r (\bsalpha_\ell)_{1:2s} = \bseta$ and thus
$\prod_{\ell=1}^r |(\bsalpha_\ell)_{s+1:2s}|! \le |\bseta|!$.

It remains to bound the factor $T_2(\bsy)$. Combining the
notations \eqref{eq:bj} and \eqref{eq:phi-qoi}--\eqref{eq:phi-sum} with
\eqref{eq:LaxMil-G}--\eqref{eq:QoI2}, we can write the a priori bounds for
$\phibar$ and $\varphi_i$ as
\begin{align}\label{eq:bnd_qoi}
	|\phibar(\bsz)|
	\le  \frac{\cbar\,\|\calG\|_{V^*}}{\amin(\bsz)}\,, \quad
	|\varphi_i(\bsz)|
	\le  \frac{c_i\,\|\calG\|_{V^*}}{\amin(\bsz)}\,,
	\quad \text{and thus} \quad
   \frac{c_0\|\,\calG\|_{V^*}}{\varphi_0(\bsz)\,a_{\min}(\bsz)} \ge 1.
\end{align}
The last inequality in \eqref{eq:bnd_qoi} allows us to use 
$r\le 2|\bseta| + q - 1$
to bound $T_2(\bsy)$ independently of~$r$ by
\begin{equation}
\label{eq:T2bnd1}
T_2(\bsy) \leq  \frac{1}{\varphi_0(\bsz)^{2|\bseta| + 2q - 1}}
 \bigg(\frac{\max(S(\bsy),c_0,1)\,\|\calG\|_{V^*}}{a_{\min}(\bsz)}\bigg)^{2|\bseta| + q - 1}.
\end{equation}

It follows from \eqref{eq:xi-explicit}
that
\begin{align*}
 |\xi(t,\bsy)|
 \le \frac{|t|+|\bar\varphi(\bsz)|+\sum_{i=1}^s |\varphi_i(\bsz)|\,|w_i|}{\varphi_0(\bsz)}
 \le \frac{|t| + (\bar c +\sum_{i=1}^s c_i \,|w_i|) \frac{\|\calG\|_{V^*}}{a_{\min}(\bsz)} }{\varphi_0(\bsz)},
\end{align*}
which leads to
\begin{align*}
 S(\bsy)
 &\le \bar{c} + c_0 \frac{|t| + (\bar{c} + \sum_{i=1}^s c_i \,|w_i|) \frac{\|\calG\|_{V^*}}{a_{\min}(\bsz)} }{\varphi_0(\bsz)}
 + \sum_{i=1}^s c_i \,|w_i| \\
 &= \frac{c_0\,|t|}{\varphi_0(\bsz)}
 + \bigg(\frac{c_0}{\varphi_0(\bsz)}\frac{\|\calG\|_{V^*}}{a_{\min}(\bsz)} + 1\bigg)
 \bigg(\bar{c} + \sum_{i=1}^s c_i \,|w_i|\bigg).
\end{align*}
Adding $c_0+1$ after $\cbar$ and using the last inequality in
\eqref{eq:bnd_qoi} again yields
\begin{align} \label{eq:maxMbd}
 \max(S(\bsy), c_0, 1)
 &\le \frac{c_0}{\varphi_0(\bsz)} \bigg(|t| + \frac{2\|\calG\|_{V^*}}{a_{\min}(\bsz)}
 \bigg(\bar{c} + c_0 + 1 + \sum_{i=1}^s c_i \,|w_i|\bigg)
 \bigg).
\end{align}
Combining \eqref{eq:hT1T2bd}, \eqref{eq:T1bnd}, \eqref{eq:T2bnd1}, and
\eqref{eq:maxMbd} gives the final bound \eqref{eq:h-bound}.
\end{proof}

The following Lemma shows that the bound \eqref{eq:h-int} holds with an explicit constant
$B_{q,\bseta}$. In the proof, we will take a more liberal upper bound on
$|h_{q,\bseta}(t,\bsy)|$ so that it can be integrated easily.

\begin{lemma} \label{lem:B}
Let $\varphi$ be the quantity of interest \eqref{eq:phi-qoi} for the PDE
problem \eqref{eq:varpde}. Given $q\in\bbN_0$ and $\bseta\in\bbN_0^{2s}$, 
define
\begin{align} \label{eq:B-def}
 B_{q,\bseta}
 &\coloneqq A_{q,\bseta} \frac{(|\bseta| + q - 1)!\, (|\bseta|!)^2}{(\ln 2)^{2|\bseta|}}
 \bigg(\prod_{j=1}^s b_j^{2\eta_{j+s}} \bigg)
 \bigg(\prod_{i=1}^s c_i^{2\eta_{i}} \bigg)
\end{align}
where
\begin{align}
\label{eq:Aq,eta}
 A_{q,\bseta}
 &\coloneqq  
\frac{1}{2\pi}
 \bigg(\frac{\ell_{0,\inf} + \|u_0(\cdot,\bszero)\|_{W^{1,\infty}(D)}}{\varphi_0(\bszero)\,\ell_{0,\inf}}\bigg)^{4|\bseta|+2q}
\nonumber\\
 &\qquad\cdot
 \bigg(c_0\|\calG\|_{V^*}\big(\max(|t_0|,|t_1|) +2\|\calG\|_{V^*}(\bar{c} + c_0+ 1)\big)\bigg)^{2|\bseta|}
\nonumber\\
 &\qquad\cdot
 \bigg(\prod_{\satop{i=1}{\eta_i\ne 0}}^s I_\psi \big( (2|\bseta| + q - 1)c_i \big) \bigg)
 \bigg(\prod_{\satop{i=1}{\eta_i= 0}}^s I_\rho \big( (2|\bseta| + q - 1) c_i \big) \bigg) 
 \nonumber\\&\qquad \cdot
 \bigg(\prod_{\satop{j=1}{\eta_{s+j} \neq 0}}^s \!\!\!
 I_\psi \big( 2(6|\bseta|+4q - 3) \widehat{b}_j \big) \bigg)
 \bigg(\prod_{\satop{j=1}{\eta_{s+j}= 0}}^s \!\!\!
 I_\rho \big( 2(6|\bseta|+4q - 3)\widehat{b}_j \big) \bigg),
\end{align}
with $\widehat{b}_j \coloneqq \|a_j\|_{W^{1,\infty}(D)} \ge \|a_j\|_{L^\infty(D)} = b_j$,
\begin{align} 
\label{eq:I1I2}
 I_\psi(\Theta) &\coloneqq \int_{-\infty}^\infty e^{2 \Theta |y|}\,\psi(y) \,\rd y
 \quad\mbox{and}\quad
 I_\rho(\Theta) \coloneqq \int_{-\infty}^\infty e^{2 \Theta |y|}\,\rho(y) \,\rd y.
\end{align}
Suppose that the weight function $\psi$ is such that the integrals
$I_\psi(\cdot)$ in \eqref{eq:Aq,eta} are finite, then
the function
$h_{q,\bseta}$ of the form \eqref{eq:h-form} satisfies \eqref{eq:h-int}
with $B_{q, \bseta}$ defined above.
\end{lemma}

\begin{proof}

Here we continue
from the proof of Lemma~\ref{lem:h_bound} to obtain a more liberal upper
bound on~$T_2(\bsy)$.
To simplify the notation, using $\dottimes$ to denote componentwise multiplication of two vectors and $|\bsv| \coloneqq \sum_j |v_j|$ we abbreviate
\[
  |\widehat\bsb\dottimes \bsz| = \sum_{j = 1}^s |\widehat{b}_j z_j| \,=\, \sum_{j = 1}^s \widehat{b}_j|z_j|, \quad
  |\bsb\dottimes\bsz| = \sum_{j = 1}^s b_j|z_j|
  \quad\mbox{and}\quad
  |\bsc\dottimes\bsw| = \sum_{i=1}^s c_i |w_i|.
\]

From the definition \eqref{eq:axz} it follows that
\[
 \amin(\bsz) \ge e^{-|\bsb\dottimes\bsz|}
 \quad\mbox{and}\quad
 \amax(\bsz) \le e^{|\bsb\dottimes \bsz|}.
\]
Using also $1\le e^x$ and $1 + x \le e^x$ for $x\ge 0$, we have from
\eqref{eq:maxMbd} that
\begin{align*}
  \max(S(\bsy), c_0, 1)
 &\le \frac{c_0}{\varphi_0(\bsz)} \Big( |t| + 2\|\calG\|_{V^*} (\bar{c}+c_0+1)\Big)
 e^{|\bsb\dottimes\bsz| + |\bsc\dottimes\bsw|}, 
\end{align*}
and thus, from \eqref{eq:T2bnd1} we have
\begin{align*}
 T_2(\bsy) \le 
 \frac{e^{(2|\bseta| + q - 1)(2|\bsb \dottimes \bsz| + |\bsc\dottimes\bsw|)}}
 {\varphi_0(\bsz)^{4|\bseta|+3q - 2}}
 \big[c_0\|\calG\|_{V^*}
 \big( |t| + 2\|\calG\|_{V^*} (\bar{c}+c_0+1)\big) \big]^{2|\bseta| + q - 1}
 .
\end{align*}
Next, the lower bound on $\varphi_0(\bsz)$ from Corollary~\ref{cor:lower}
together with \eqref{eq:K0} gives
\begin{align*}
 \frac{1}{\varphi_0(\bsz)}
 &\le \frac{e^{|\bsb\dottimes\bsz|}}{\varphi_0(\bszero)}
 \bigg( e^{|\bsb\dottimes\bsz|}
 + \frac{\|u_0(\cdot,\bszero)\|_{W^{1,\infty}(D)}}{\ell_{0,\inf}}\,|\widehat\bsb\dottimes\bsz|
 \bigg) \\
 &\le
  e^{2|\widehat\bsb\dottimes\bsz|}
  \,\bigg(\frac{\ell_{0,\inf} + \|u_0(\cdot,\bszero)\|_{W^{1,\infty}(D)}}{\varphi_0(\bszero)\,\ell_{0,\inf}} \bigg) ,
\end{align*}
and hence,
\begin{align}
\label{eq:T2bnd2}
 T_2(\bsy) &\le e^{2(6|\bseta|+4q - 3) |\widehat\bsb\dottimes \bsz| + (2|\bseta| + q - 1) \,|\bsc\dottimes \bsw|}
 \,\bigg(\frac{\ell_{0,\inf}  + \|u_0(\cdot,\bszero)\|_{W^{1,\infty}(D)}}{\varphi_0(\bszero)\,\ell_{0,\inf}}
 \bigg)^{4|\bseta|+3q - 1}
  \nonumber\\
 &\qquad \cdot\Big(c_0\|\calG\|_{V^*}
 \big(|t| + 2\|\calG\|_{V^*}(\bar{c}+c_0+1)\big)\Big)^{2|\bseta| + q - 1} ,
\end{align}
which replaces the second line in the upper bound \eqref{eq:h-bound}.

To compute the integral \eqref{eq:h-int}, we square the bound on $|h_{q,\bseta}(t,\bsy)|$ then integrate with respect to $\bsy$ against
the product $\bspsi_\bseta \cdot \bsrho_{-\bseta}$.
First, note that the terms on the first line of the bound \eqref{eq:h-bound}
are independent of $\bsy$, giving the factorial and product factors 
in $B_{q, \bseta}$ from \eqref{eq:B-def}, with the leading $1/\sqrt{2\pi}$
factor incorporated into $A_{q,\bseta}$ as defined in \eqref{eq:Aq,eta}.
Collecting the $\bsy$-independent factors from \eqref{eq:T2bnd2} gives the next two factors in $A_{q,\bseta}$.
Finally, the integral factor comes from the exponential in \eqref{eq:T2bnd2}
and is given by
\[
 \int_{\bbR^{2s}} e^{2(6|\bseta|+4q - 3) |\widehat\bsb\dottimes \bsz| + (2|\bseta| + q - 1) \,|\bsc\dottimes \bsw|}
\,\bspsi_{\bseta}(\bsy_\bseta)\,\bsrho_{-\bseta}(\bsy_{-\bseta})\,\rd\bsy,
\]
which expands to become products of univariate integrals of the form
\eqref{eq:I1I2}, leading to the last four product factors in $A_{q, \bseta}$
and yielding the formula \eqref{eq:B-def} for $B_{q,\bseta}$.
\end{proof}

We have $I_\rho(\Theta) = 2 e^{2\Theta^2} \Phi(2\Theta)$, where $\Phi$ is
the cumulative distribution function of $\rho$, and
\begin{align*}
 I_\psi(\Theta)
 = \begin{cases}
 \displaystyle\frac{1}{\mu-\Theta} & \mbox{for exponential weight function \eqref{eq:psi-exp} with $\mu>\Theta$}, \\
 \displaystyle {2}\sqrt{\frac{\pi}{\mu}}\exp\bigg(\frac{\Theta^2}{\mu}\bigg)
 {\Phi\bigg(\sqrt{\frac{2}{\mu}}\Theta\bigg)}
 & \mbox{for Gaussian weight function \eqref{eq:psi-Gauss} with $\mu<1/2$}.
 \end{cases}
\end{align*}

In \eqref{eq:Aq,eta} we have factors such as $I_\psi\big(2(6|\bseta|+4q - 3)\widehat{b}_j\big)$ that must be finite. If we use the exponential weight function, then we
would require $\mu > 2(6|\bseta|+4q - 3)\widehat{b}_j$ for all $j$,
which would be too large, resulting in an impractical, peaky weight function.
In particular, for the case $\bsnu = \bsone$,
as required for the QMC theory from Section~\ref{sec:qmc}, 
the maximum order of 
$\bseta \leq \bsone$ is $2s$ and the condition on $\mu$ becomes 
$\mu > 24s + {2}$ (if we suppose additionally that $\max_j \widehat{b}_j = 1$ and $q = 1$).
As such, in our computations we will take the Gaussian weight function given by \eqref{eq:psi-Gauss} for which the result holds for all $\mu< 1/2$. Recall that the requirement $\mu<1/2$ has additional importance as it is needed to ensure condition \eqref{eq:psi} holds.

The above lemmas finish verifying the stipulations of Assumption~\ref{asm:phi-h}, thereby ensuring that the bounds on the norms of $g_\mathrm{cdf}$ in \eqref{eq:norm-cdf} and $g_\mathrm{pdf}$ in \eqref{eq:norm-pdf} hold. 

\section{Error analysis}
\label{sec:error}

Having verified that the PDE quantity of interest \eqref{eq:qoi}
satisfies the necessary conditions for the smoothing by preintegration theory,
we now present bounds for the error of the QMC approximations of the
distribution and density functions.

First, we recall the bounds on the norms for the preintegrated expressions for the cdf and pdf. We restrict ourselves to $\calH_{2s}^\bsone$,
the Sobolev space of first-order dominating mixed smoothness, since that is what is necessary to apply the CBC error bound \eqref{eq:qmc-err}. 
The following bounds will
apply to all points in the interval $t\in[t_0,t_1]$  over which we
evaluate the pdf and cdf. 
Following from equations \eqref{eq:norm-cdf} and \eqref{eq:norm-pdf}, the upper bounds on the norms for the preintegrated function $g$, in both the cdf and pdf case, can be expressed generally as
\begin{align}\label{eq:norm-general}
 \|g\|_{\calH^{\bsone}_{2s}} \leq \Bigg(\sum_{\bseta\leq \bsone}
 \frac{\Lambda_{\bseta}}{\gamma_{\bseta}}\Bigg)^{1/2},
 \quad
 \Lambda_{\bseta}=
 \begin{cases}
 1 & \mbox{if $g = g_{\rm cdf}$ and } \bseta = \bszero\\
 \big(3^{|\bseta|-1}(|\bseta|-1)!\big)^2 B_{0, \bseta} & \mbox{if $g = g_{\rm cdf}$ and } \bseta \neq \bszero\\
 \big(3^{|\bseta|}|\bseta|!\big)^2 B_{1, \bseta} & \mbox{if $g = g_{\rm pdf}$},
 \end{cases}
\end{align}
where the factor $8^{|\bseta|}$ in \eqref{eq:norm-cdf} and \eqref{eq:norm-pdf} is now replaced by $3^{|\bseta|}$, due to additional simplification that can be done in the proofs of \cite[Theorems 3.2 and 3.3]{GKS23} as a result of the linear source term we consider.

The following theorem presents the main result on the error of our QMC with preintegration approximations of the cdf and pdf, for the case of a Gaussian weight function.

\begin{theorem}\label{thm:qmc_rate} Let Assumption \ref{asm:pde} hold. Let $F$ and $f$ be the cdf and pdf of the quantity of interest $\calG(u(\cdot,\bsy))$, where $\calG\in V^*$ is a positive linear functional and $u$ is the solution to \eqref{eq:weak_form}. For $\epsilon \in(0,1/2]$, consider Gaussian weight functions $\psi$  defined by \eqref{eq:psi-Gauss} with parameter $\mu\in(0,\epsilon)$.
For $N\in\N$, where $N$ is a prime power, a single randomly shifted lattice rule can be constructed for both the approximation of the cdf $F_N(t)$ as in \eqref{eq:F_N} and pdf $f_N(t)$ as in \eqref{eq:f_N}, for all $t\in[t_0,t_1]$,
such that the corresponding errors satisfy
\begin{align}
\sqrt{\bbE_{\bsDelta}[|F(t)-F_N(t)|^2]}\, &\leq\,  C_{\bsgamma,s,\epsilon}\, N^{-1+\epsilon} \quad \text{and} \label{eq:qmc-err-F}\\
\sqrt{\bbE_{\bsDelta}[|f(t)-f_N(t)|^2]}\, &\leq\, C_{\bsgamma,s,\epsilon}\, N^{-1+\epsilon}, \label{eq:qmc-err-f}
\end{align}
where $C_{\bsgamma,s,\epsilon}<\infty$ is independent of $N$, but depends on $s$.

The weight parameters $\bsgamma=(\gamma_\bseta)_{\bseta\leq \bsone}$ that minimise the constant $C_{\bsgamma,s,\epsilon}$ are
\begin{align*}
\gamma_{\bseta}^* = \bigg[\frac{\Gamma_{|\bseta|}}{\varrho(\epsilon)^{|\bseta|}}\, 
 \bigg(\prod_{j=1}^s b_j^{2\eta_{j+s}} \bigg)
 \bigg(\prod_{i=1}^s c_i^{2\eta_{i}} \bigg)\Bigg]^{\frac{2(1-\epsilon)}{3-2\epsilon}},
\end{align*}
where $\Gamma_{|\bseta|}$ is defined below in \eqref{eq:Gamma_m}, $b_j,c_i$ are defined in \eqref{eq:bj} and $\varrho$ is given by \eqref{eq:qmc-err-const-Gauss}.
\end{theorem}

\begin{proof}

We start by bounding $\Lambda_\bseta$ in  \eqref{eq:norm-general} so that the RMSE of the QMC approximation for both the cdf and pdf can be bounded by the same expression. This upper bound on $\Lambda_{\bseta}$ is 
\begin{align*}
 \overline{\Lambda}_{\bseta} \,\coloneqq\, 
 \Gamma_{|\bseta|} \bigg(\prod_{j=1}^s b_j^{2\eta_{j+s}} \bigg)
\bigg(\prod_{i=1}^s c_i^{2\eta_{i}} \bigg) \quad \text{for} \quad \bseta \leq \bsone,
\end{align*}
where we define 
	\begin{align}\label{eq:Gamma_m}
	\Gamma_{m} &\,\coloneqq\, \frac{2^{s-1}}{\pi} {\bigg(\frac{9\sqrt{\pi}}{(\ln2)^2\sqrt{\mu}}\bigg)^m}
	\max\left\{1,\bigg(\frac{\ell_{0,\inf} + \|u_0(\cdot,\bszero)\|_{W^{1,\infty}(D)}}{\varphi_0(\bszero)\,\ell_{0,\inf}}\bigg)\right\}^{4m+2} \notag\\
	  &\qquad\cdot\bigg(c_0\|\calG\|_{V^*}\big(\max(|t_0|,|t_1|) +2\|\calG\|_{V^*}(\bar{c} + c_0+ 1)\big)\bigg)^{2m}\notag\\
	  &\qquad\cdot (m!)^5 \,\exp\bigg(\frac{4m^2}{\mu} \sum_{i=1}^{s} c_i^2\bigg)\exp\bigg(\frac{(12m+2)^2}{\mu}\sum_{j=1}^{s} \widehat{b}_j^2\bigg)
	\end{align}
and $\widehat b_j = \|a_j\|_{W^{1,\infty}(D)}$. {Note that we arrive at $\Gamma_m$ by combining all terms depending on $|\bseta|$ explicitly while setting $q=1$ in $B_{q,\bseta}$ and $A_{q,\bseta}$ defined by \eqref{eq:B-def} and \eqref{eq:Aq,eta}, respectively}. We have also bounded the products involving $c_i$ in $A_{q, \bseta}$ from above by
	\begin{align*}
	\bigg(\prod_{\satop{i=1}{\eta_i\ne 0}}^s I_\psi \big( (2|\bseta| + q - 1)c_i \big) \bigg)
 \bigg(\prod_{\satop{i=1}{\eta_i= 0}}^s I_\rho \big( (2|\bseta| + q - 1) c_i \big) \bigg) 
 \leq 2^{s}{\bigg(\sqrt{\frac{\pi}{\mu}}\bigg)^{|\bseta|}} \exp\bigg(\frac{4|\bseta|^2}{\mu} \sum_{i=1}^{s} c_i^2\bigg)
	\end{align*}
by  using the bounds $I_\psi(\Theta) \leq 2\sqrt{\pi/\mu}\,e^{\Theta^2/\mu}$ and $I_\rho(\Theta) {\leq} 2e^{2\Theta^2} \leq 2e^{\Theta^2/\mu}$ for $0<\mu<1/2$. A similar argument is used to bound the product of integrals involving $\widehat b_j$.
	
Applying Theorem \ref{thm:RMSE-bound-g} with $\overline\Lambda_{\bseta}$
gives a generic form for the RMSE bound for the cdf,
\begin{align}\label{eq:RMSE-g}
		\sqrt{\bbE_{\bsDelta}[|F(t)-F_N(t)|^2]} \notag
		&= \sqrt{\bbE_{\bsDelta}[|I_{2s}(g_\mathrm{cdf})-Q_{2s,N}(g_\mathrm{cdf})|^2]} \notag \\
		&\leq
		\bigg(\frac{1}{\phi_{\rm{tot}}(N)}
		\sum_{\bszero\neq\bseta\leq\bsone}\gamma_\bseta^{1/(2(1-\epsilon))}\,\varrho(\epsilon)^{|\bseta|}
		 \bigg)^{1-\epsilon}
		\left(\sum_{\bseta\leq \bsone} \frac{\overline\Lambda_{\bseta}}{\gamma_{\bseta}}\right)^{1/2},
	\end{align}
and the same bound holds for the RMSE of $f_N$ for the pdf case.

From \cite[Lemma 6.2]{KSS12}, the  weight parameters that minimise the upper bound \eqref{eq:RMSE-g} are
	\begin{align*}
		\gamma_\bseta^*
		= \Bigg(\frac{\overline\Lambda_{\bseta}}{\varrho(\epsilon)^{|\bseta|}}\Bigg)^{\frac{2(1-\epsilon)}{3-2\epsilon}},
	\end{align*}
which can be substituted into \eqref{eq:RMSE-g} to give the corresponding bound on the RMSE,
\begin{align*}
			\sqrt{\bbE_{\bsDelta}[|F(t)-F_N(t)|^2]}  \leq 
			\Bigg(\sum_{\bseta\leq\bsone} 
			\left(\varrho(\epsilon)^{|\bseta|}\,\overline\Lambda_{\bseta}^{1/(2(1-\epsilon))}\right)^{\frac{2(1-\epsilon)}{3-2\epsilon}} \Bigg)^{3-2\epsilon}[\phi_{\rm{tot}}(N)]^{-1+\epsilon}.
		\end{align*} 
Since $N$ is a prime power, we have $1/\phi_{\rm{tot}}(N) \leq 2/N$, 
which can be substituted into the bound above to give the required results. 
\end{proof}

\begin{remark}
We emphasise that the convergence rate of $\calO(N^{-1+\epsilon})$ for arbitrarily small $\epsilon$ proven above is the same rate obtained when 
computing the expected value of the quantity of interest using QMC,
as in, e.g., \cite{GrKNSSS15}. 
Considering the added difficulty of approximating distributions and densities, due to the discontinuities in the integrals \eqref{eq:cdf-y0} and \eqref{eq:pdf-y0}, this is a significant result. 
We note that the implied constants for the two problems, however, are not the same, with the constant here depending on dimension.
The dependence on dimension is an artefact of the proof, resulting from the bounds on the norms of $g_\mathrm{cdf}$ and $g_\mathrm{pdf}$, which depend non-linearly on the quantity of interest $\varphi$, see \eqref{eq:g-xi}, and so require the chain rule to compute the derivatives in the norms. 
Dimension dependence in the implied constant is the price of tackling a much more difficult integrand while still obtaining with same convergence rate as the simpler problem.
\end{remark}

\section{Numerical results}\label{sec:num}
In this section we present the results of numerical experiments on approximating the cdf and pdf. We consider the spatial domain $D = (0,1)^2$ with $s=64$, i.e., the full parametric dimension is $2s+1= 129$. Solutions to the PDE are obtained using a FEniCS \cite{fenics1,fenics2, fenics3, fenics4} finite element solver with mesh width $h=2^{-8}$. The experiments are conducted using the computational cluster Katana \cite{Katana} supported by Research Technology Services at UNSW Sydney. 

In \eqref{eq:lxw} and  \eqref{eq:axz}, we set  $\ellbar \equiv \ell_0 \equiv 1$, 
	\begin{align*}
	\ell_{i}(\bsx) &= \frac{1}{1+(i\pi)^\theta}\sin(i\pi x_1)\sin((i+1)\pi x_2)
	\quad \,\,\text{for} \quad i=1,\ldots,64, \mbox{ and }\\
	a_j(\bsx) &= \frac{\alpha}{1+(j\pi)^\theta}\sin(j\pi x_1)\sin((j+1)\pi x_2)
	\quad \text{for}\quad j=1,\ldots,64,
	\end{align*}
where $\alpha>0$ is the scaling parameter controlling the level of influence of the lognormal random coefficient and $\theta>0$ is the decay parameter controlling how quickly the importance of each random parameter decreases. The linear functional $\calG$ used for the results is point evaluation of the solution at $\bsx = (1/\sqrt{2},1/\sqrt{2})$.

We construct our lattice rule with a modified version of the weights $\bsgamma$ presented in Theorem~\ref{thm:qmc_rate}. The weights in Theorem \ref{thm:qmc_rate} are of POD (product and order-dependent) form \cite{KSS12}, however, as the magnitude of the order-dependent part grows too quickly with dimension, we will only include the factorial component of the order-dependent part. We used liberal bounds in our derivations in Section \ref{sec:error} and thus the large order-dependent part is likely overstimated, motivating its modification. Recall we are using Gaussian weight functions, hence $\varrho$ will be of the form \eqref{eq:qmc-err-const-Gauss} where we take $\mu=0.05$ and $\epsilon = 2\mu=0.1$. The new weights are given by
\begin{align}\label{eq:prod_weights}
\gamma_{\bseta} = {\Gamma_{|\bseta|}}\prod_{j\in\supp(\bseta)} \gamma_j, \quad \text{where} \quad 
\gamma_j \,\coloneqq\,
\begin{dcases}
\bigg(\frac{c_j^{2}}{\varrho(\epsilon)} \bigg)^{\frac{2(1-\epsilon)}{3-2\epsilon}}
\quad \text{if} \quad j\in \{1,\ldots,s\},\\
\bigg(\frac{b_j^{2}}{\varrho(\epsilon)} \bigg)^{\frac{2(1-\epsilon)}{3-2\epsilon}} 
\quad \text{if} \quad j\in \{s+1,\ldots,2s\}\\
\end{dcases}
\end{align}
and $\Gamma_{|\bseta|}=(|\bseta|!)^5$. Note that with this change we still preserve the theoretical $\calO(N^{-1+\epsilon})$ rate of convergence shown in Theorem \ref{thm:qmc_rate} and it is only the implied constant in \eqref{eq:qmc-err-F} and \eqref{eq:qmc-err-f} that will increase accordingly. 
There is  precedence for modifying the order-dependent components of weights in \cite{KKS23}, where the  simplified product weights performed better than the theoretical weights. This is likely due to large order-dependent factors in the theoretical weights which are artefacts of the proof techniques used.

Since the CBC construction is a greedy algorithm, we choose the components of the generating vector in order of decreasing $\gamma_j$ and then permute the components back to the ordering presented in \eqref{eq:prod_weights}.

For our experiments, we  construct lattice rules for prime $N$ where
$$N\in \{503,1009,2003,4001,8009,16007,32003\},$$	
and we take $16$ random shifts to estimate the RMSE. We compare the performance of QMC and Monte Carlo (MC), both with and without preintegration. For the MC experiments, we use $16N$ samples to ensure that the results are comparable to those of the QMC experiments. We consider three choices of decays, $\theta \in \{0.1, 2, 5\}$, and compare the densities obtained for different choices of scaling parameters, $\alpha\in \{0.1,1,30\}$.

\begin{figure}[t]
    \centering
        \hspace*{-0.7cm}
        \includegraphics[scale = 0.4]{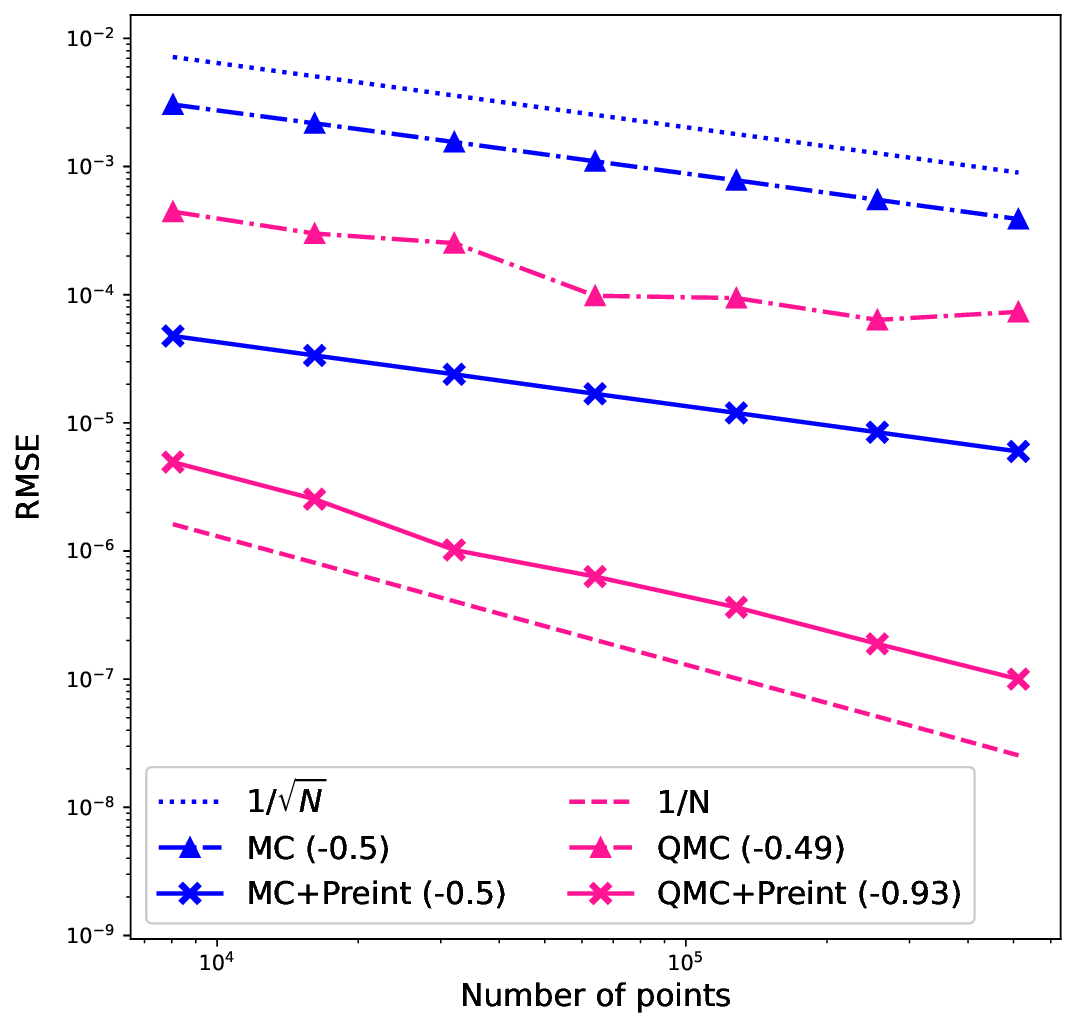}
        \includegraphics[scale = 0.4]{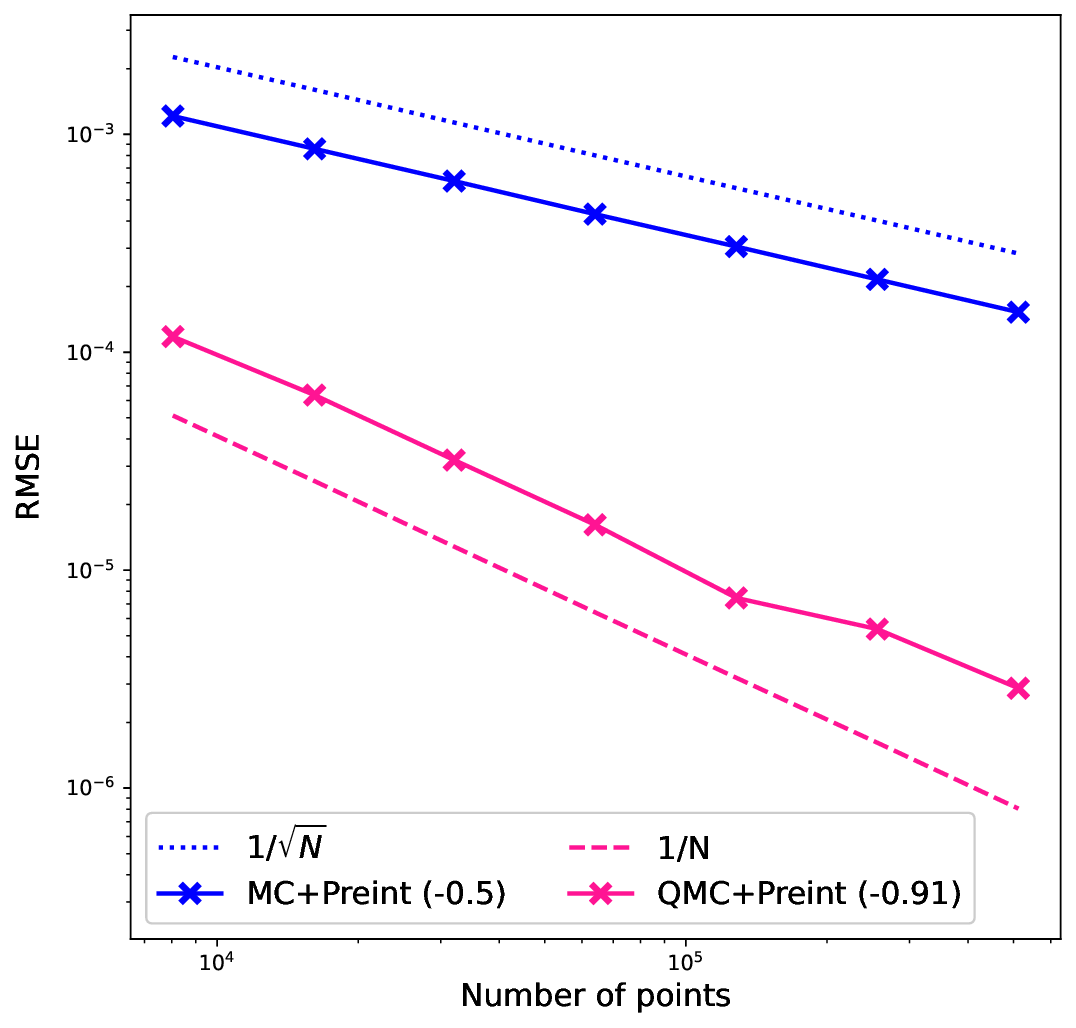}
    \caption{Convergence plots comparing MC and QMC performance (with and without preintegration) for computing the cdf (left) and pdf (right) evaluated at $t=-0.02$ with $\alpha=1$, $\theta=2$ and $16$ random shifts. The numbers in parentheses indicate the empirical rate of convergence.}
    \label{fig:a1_b1_convergence}
\end{figure}

Figure~\ref{fig:a1_b1_convergence} compares the performance of MC and QMC, both with and without preintegration, for the cdf and pdf evaluated at $t=-0.02$ with parameters $\alpha=1$ and $\theta=2$. The left plot shows the performance for the cdf problem. Notice that the results for plain MC, MC with preintegration and
plain QMC all exhibit a convergence rate of approximately $\calO(N^{-1/2})$, matching the rates predicted by the classical theory. Our QMC with preintegration computation offers the best convergence rate of close to $\calO(N^{-1})$, matching the
rate proven in Theorem~\ref{thm:qmc_rate}. As expected both QMC results are superior in absolute terms compared to their MC counterparts.

The right plot in Figure~\ref{fig:a1_b1_convergence} compares QMC 
with preintegration and MC with preintegration for the pdf with the same parameters. Since the Dirac $\delta$ cannot be evaluated explicitly, we are unable to present results for computing the pdf without preintegration. 
Again, we observe the classical rate of $\calO(N^{-1/2})$ for MC and the theoretically predicted rate of almost $\calO(N^{-1})$ for QMC.

Figure \ref{fig:cdf_pdf_a1_b1_theta2} plots the corresponding cdf and pdf for the problem with $\alpha=1$ and $\theta=2$ using $N=503$ points over the interval $t\in[-0.2,0.3]$. We also overlay the cdf and pdf histograms of 32000 randomly sampled solutions (i.e., MC) to verify whether the pdf and cdf computed using QMC and preintegration coincide with what is observed empirically. Even with as few as 503 QMC lattice points, we find a good fit. Performing a Kolmogorov--Smirnov test to assess the goodness of fit using an empirical cdf constructed from 1000 random samples, we fail to reject the null hypothesis at a 5\% significance level (i.e., we are confident the randomly sampled solutions are from the same distribution defined by our cdf computed using preintegration).

\begin{figure}[t]
    \centering
        \hspace*{-0.7cm}
        \includegraphics[scale = 0.4]{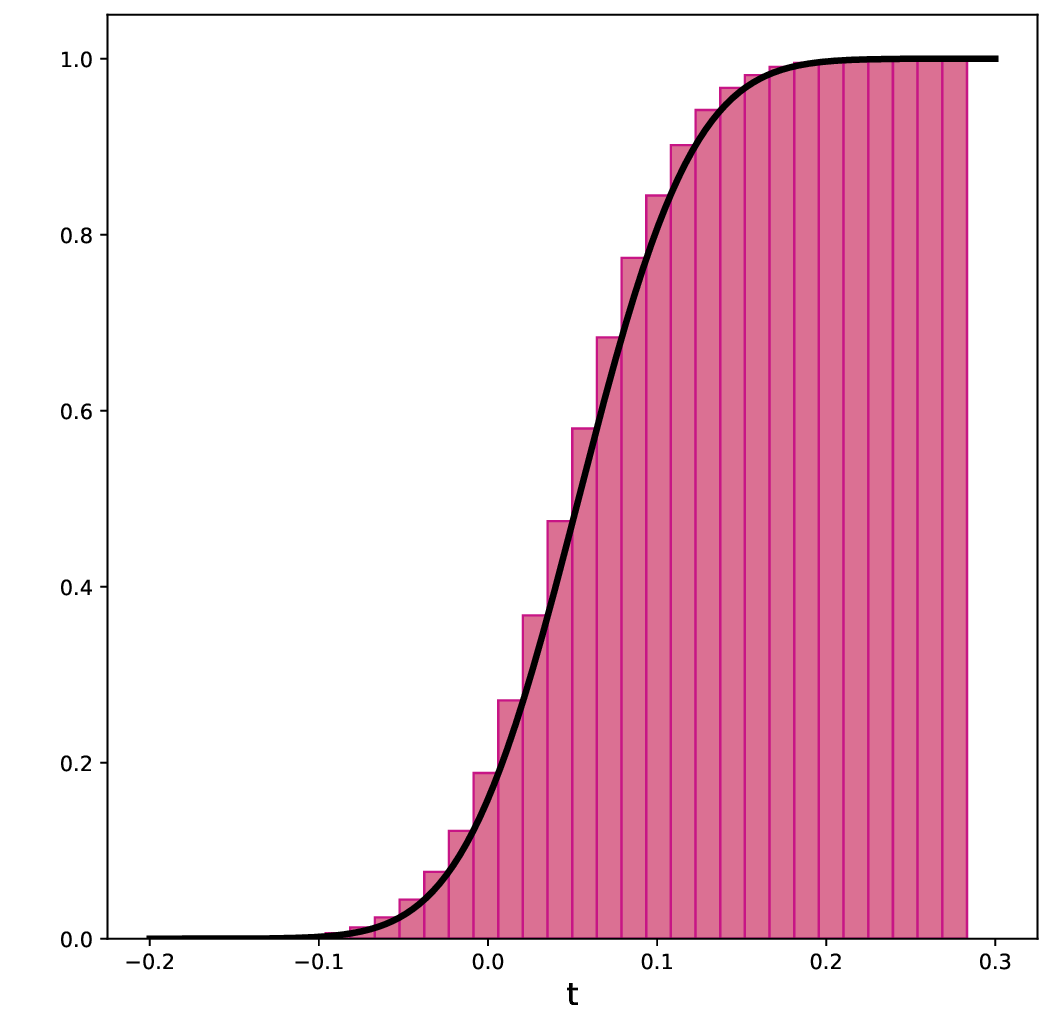}
        \includegraphics[scale = 0.4]{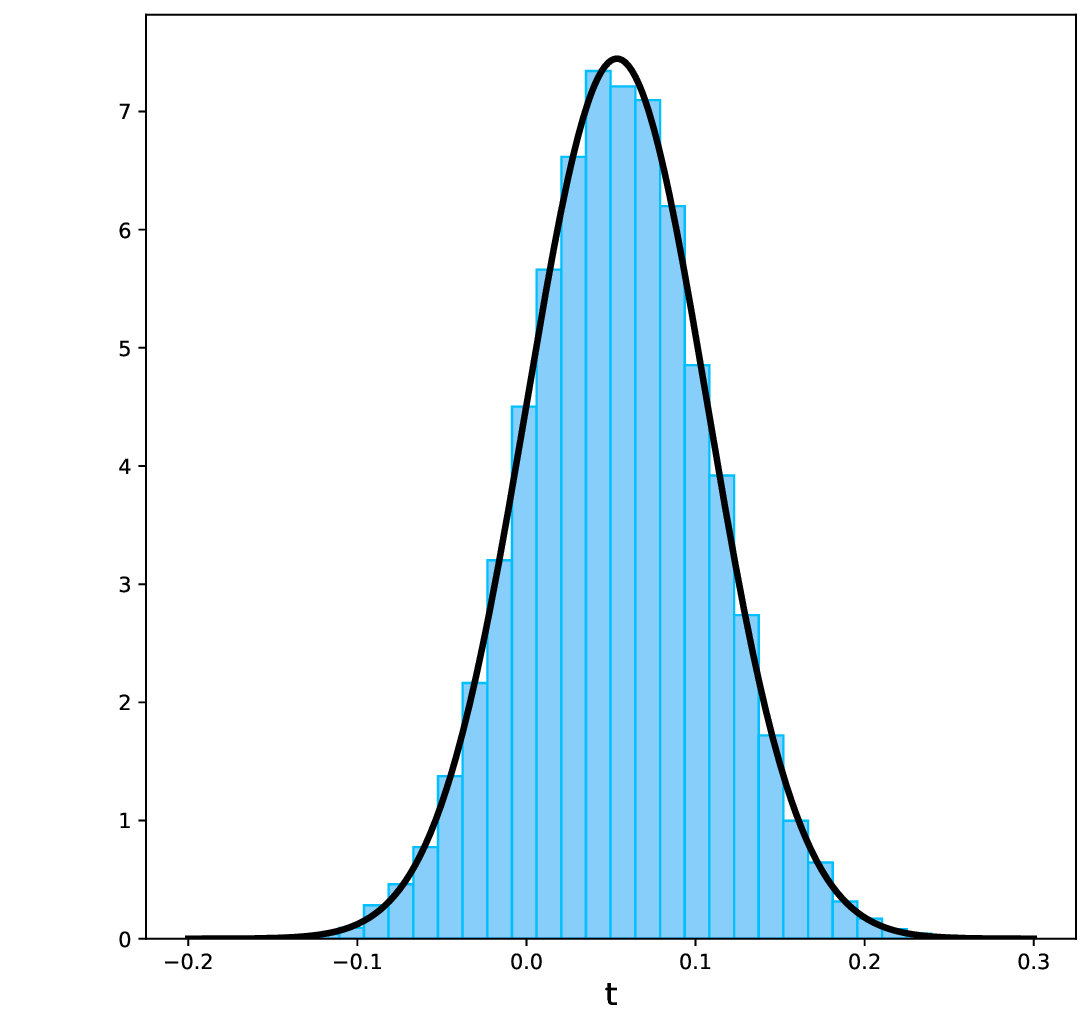}
\caption{Plots of the computed cdf (left) and pdf (right) for the problem with $\alpha=1$, $\theta=2$ and $N=503$ lattice points. Histograms of 32000 randomly sampled solutions are overlaid.}
\label{fig:cdf_pdf_a1_b1_theta2}
\end{figure}

\begin{figure}[t]
    \centering
        \hspace*{-0.7cm}
        \includegraphics[scale = 0.4]{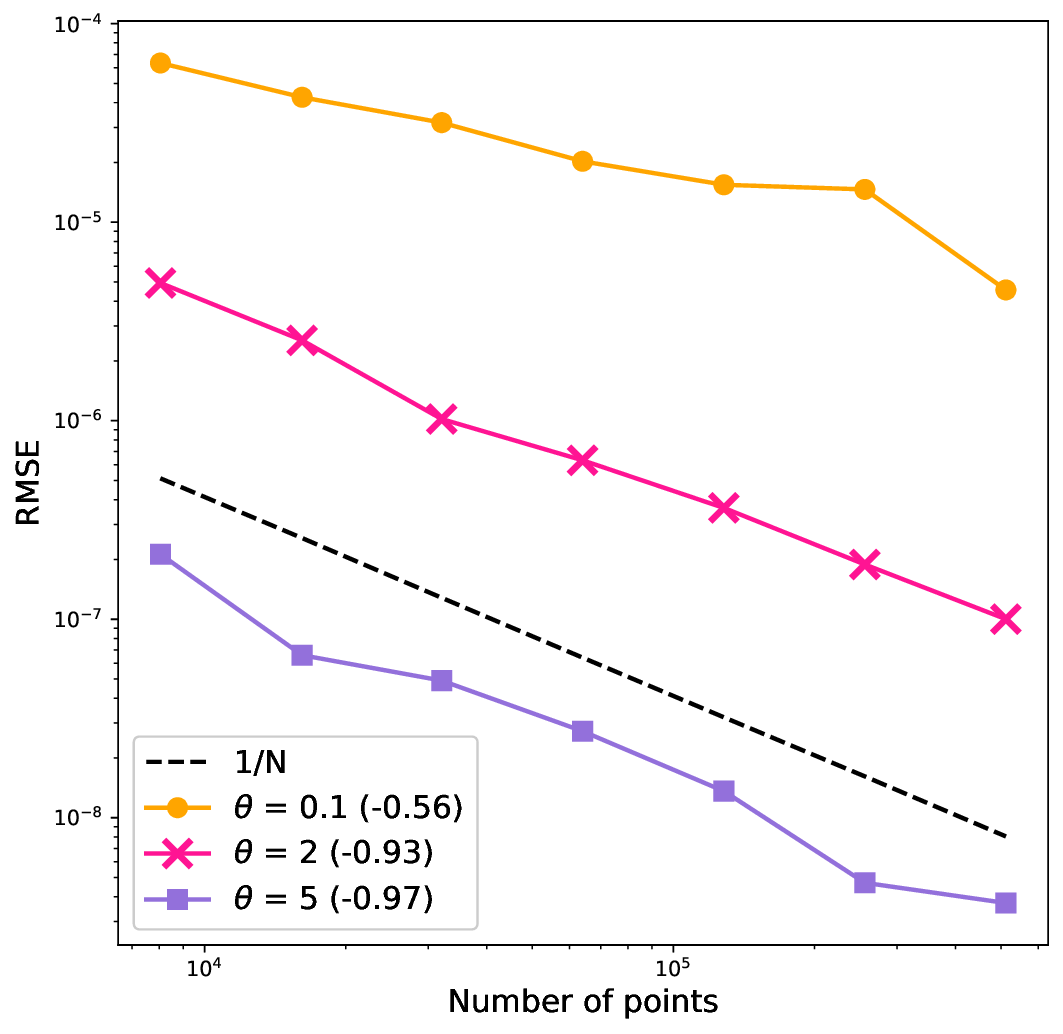}
        \includegraphics[scale = 0.4]{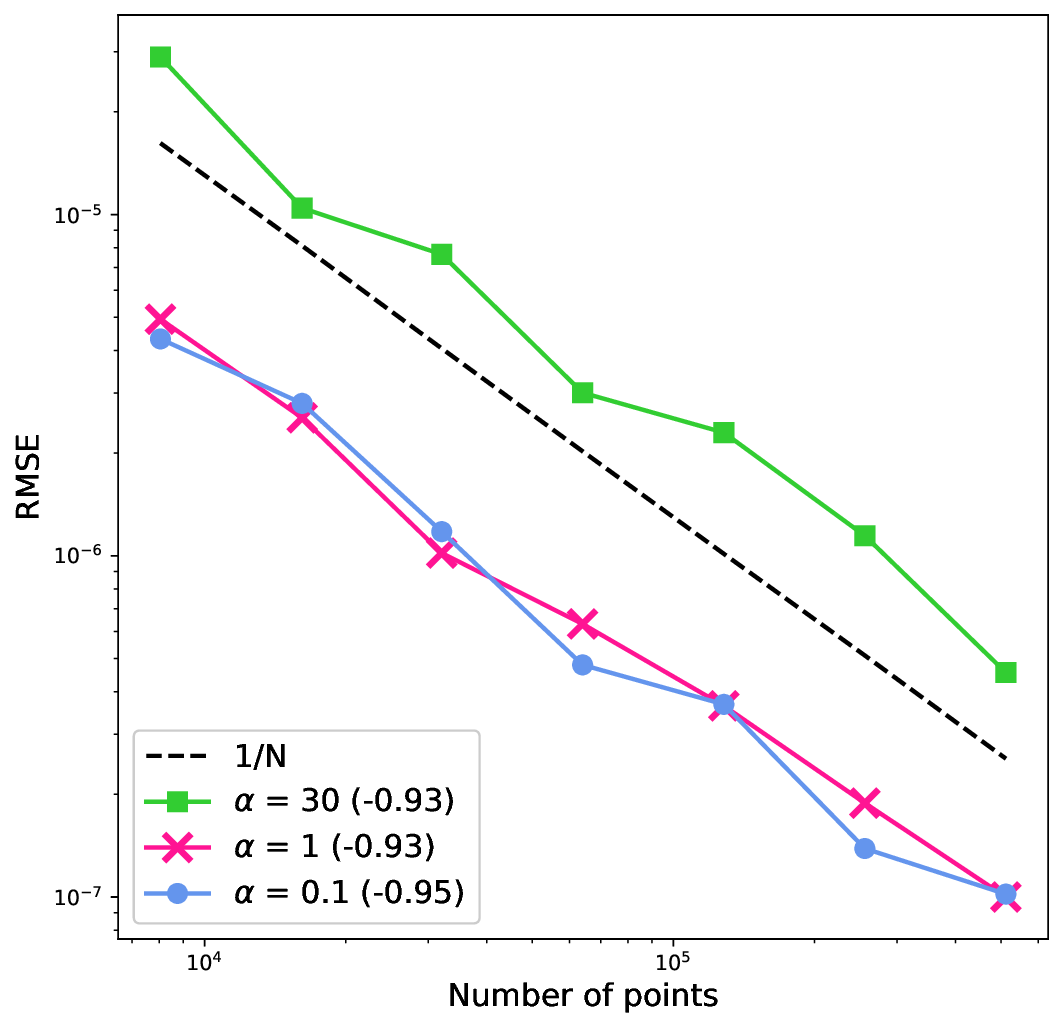}
    \caption{Convergence plots comparing QMC with preintegration performance for computing the cdf at $t=-0.02$ for fixed $\alpha=1$ and varying theta (left) and fixed $\theta=2$ and varying alpha (right) with $16$ random shifts.  The magenta lines in both plots represent the same data ($\theta = 2$ and $\alpha=1$).}    \label{fig:param_convergence}
\end{figure}

The left plot in Figure \ref{fig:param_convergence} compares the convergence for the cdf computed using QMC with preintegration for different levels of decay $\theta\in\{0.1,2,5\}$ with $\alpha=1$. For each $\theta$, we construct a new QMC rule using the weights \eqref{eq:prod_weights}. We find that the convergence rate is poor for a small decay parameter, whereas for larger decay parameters the convergence rate is close to $\calO(N^{-1})$. This is expected since the ``effective" dimension of the problem  decreases as a result of the importance of each parameter decaying faster as $\theta$ increases (see \cite{CMO97} for a precise definition of effective dimension). The magnitude of the RMSE for the three problems reflects the difficulty of the problems, with the RMSE being largest for the most difficult problem ($\theta = 0.1$) and smallest for the easiest ($\theta = 5$).

We are also interested in how the performance of QMC changes as when different random parameters become more dominant. The right plot in Figure \ref{fig:param_convergence} compares convergence when $\theta=2$ is fixed and we vary $\alpha \in\{0.1,1,30\}$. We see that the magnitude of the RMSE decreases as the impact of the lognormal random coefficient decreases. The value of $\alpha$ has no obvious impact on the rate of convergence. Note that the RMSE corresponding to $\theta = 2$ and $\alpha=1$ in Figure \ref{fig:param_convergence} are the same as the RMSE of QMC with preintegration for the CDF (left plot) in Figure \ref{fig:a1_b1_convergence}. To indicate this, these lines use the same style.

\begin{figure}[t]
\centering
        \hspace*{-0.7cm}
        \includegraphics[scale = 0.4]{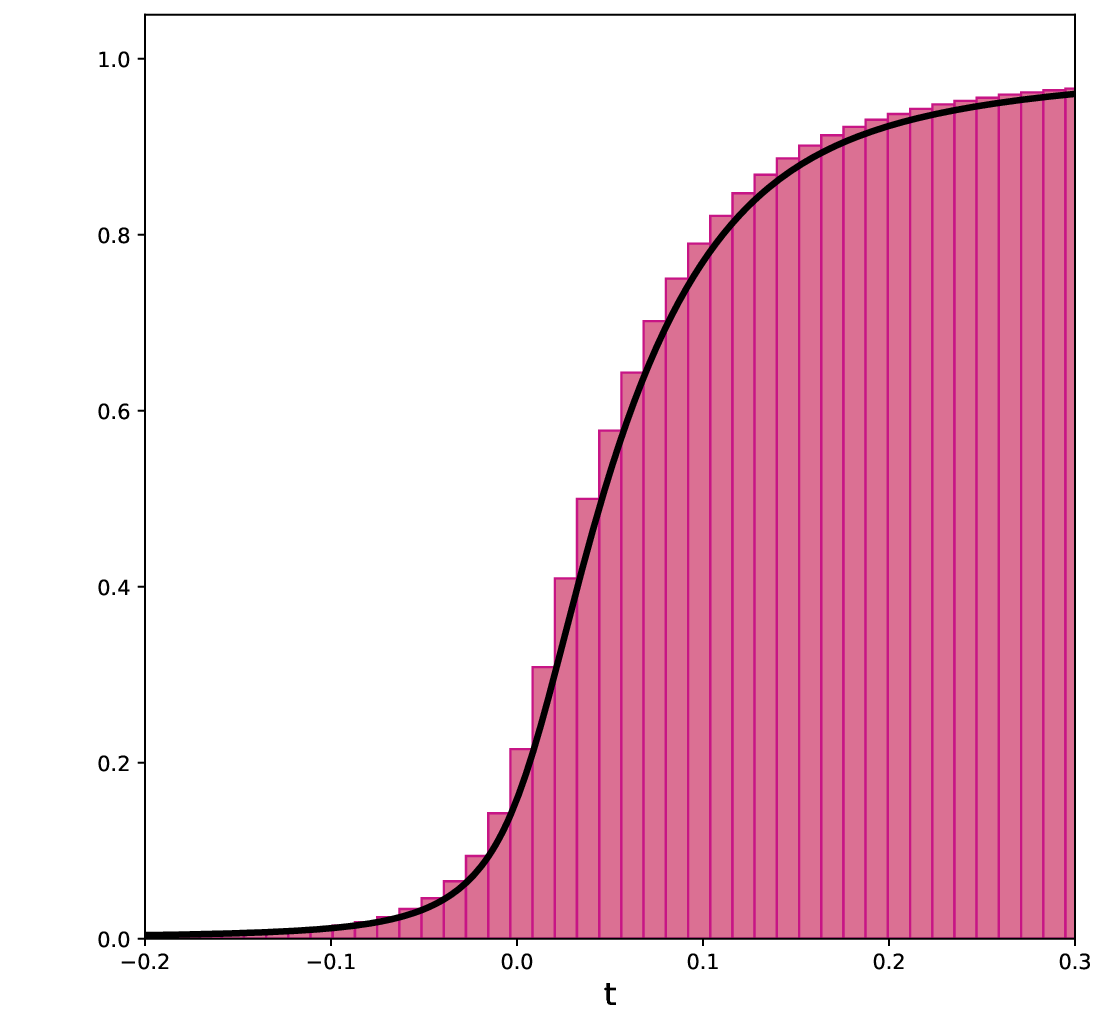}
        \includegraphics[scale = 0.4]{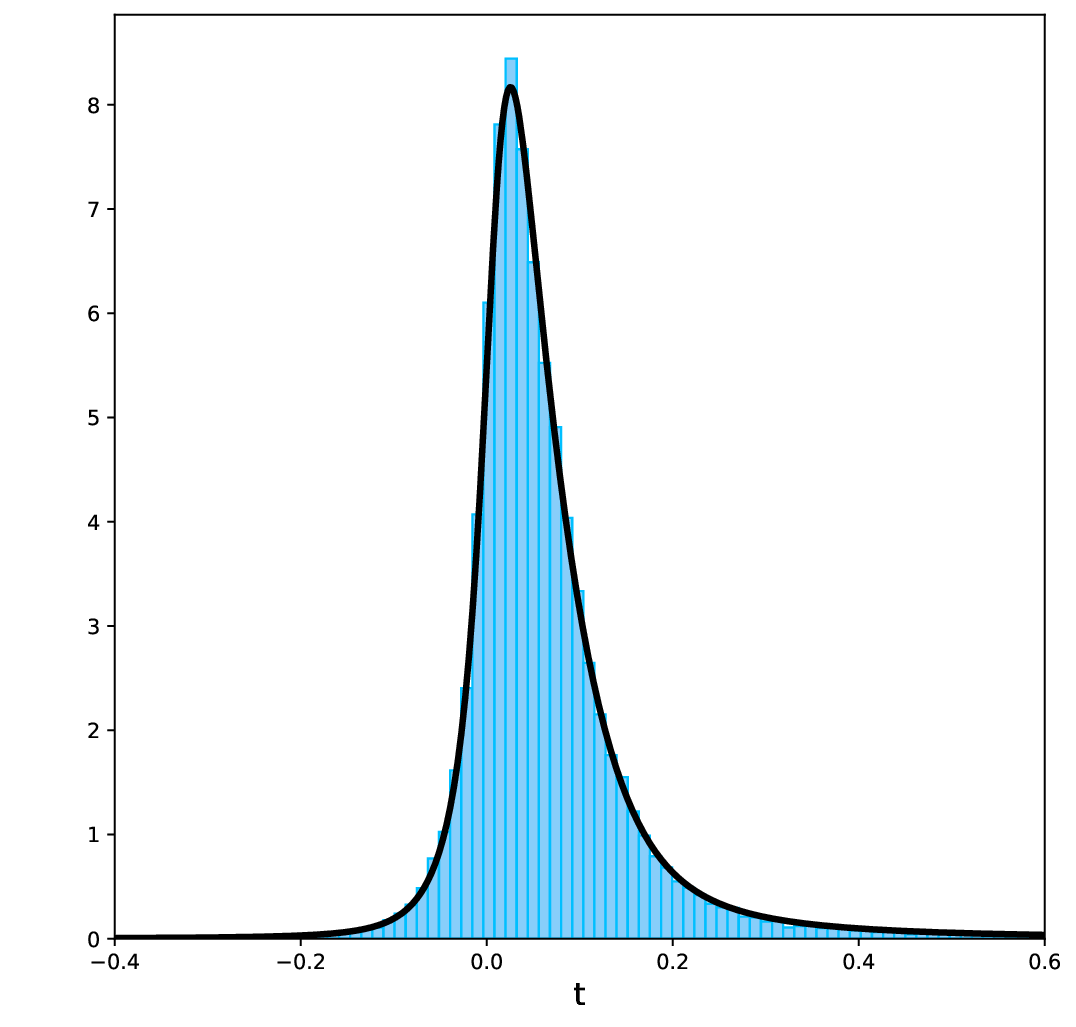}
\caption{Plots of the cdf (left) and pdf (right) for the problem with $\alpha=30$, $\theta=2$ and $N=503$ lattice points. Histograms of 32000 randomly sampled solutions are overlaid.}
\label{fig:cdf_pdf_a1_b30_theta2}
\end{figure}

Figure \ref{fig:cdf_pdf_a1_b30_theta2} shows the cdf and pdf for the case when $\alpha\gg 1$. Here we see the pdf corresponding to $\alpha =30 $ is far less symmetrical and is instead positively skewed compared to the density presented in Figure \ref{fig:cdf_pdf_a1_b1_theta2} when $\alpha=1$. This is due to the dominance of the lognormal random coefficient in comparison to the Gaussian source term resulting in heavier right tails as one would expect. In Figure \ref{fig:cdf_pdf_a1_b1_theta2}, where $\alpha = 1$, we observe a more symmetrical distribution which looks similar to the Gaussian density. 
For the case $\alpha \ll 1$ we tested $\alpha = 0.1$, in which case the cdf and pdf look very similar to those in Figure~\ref{fig:cdf_pdf_a1_b1_theta2} and thus have been been omitted.

\section{Conclusion}
\label{sec:conc}

We have presented a new method using preintegration followed by QMC to approximate the cdf and pdf of a quantity of interest from the uncertainty quantification of an elliptic PDE with a lognormally distributed coefficient and a normally distributed source term.
The key idea is to formulate the cdf and pdf as an integral of an indicator and Dirac $\delta$ function, respectively. Preintegration is used to smooth the discontinuity in the integrands and then a tailored QMC rule is applied to approximate the remaining integral.
We proved that the RMSE of our cdf and pdf approximations converge
at a rate close to $1/N$, which is the same rate as for computing the expected value of the quantity of interest.
Finally, we presented numerical results showing that the convergence in practice matches the theoretical rates and which demonstrate the superior efficiency of our method compared to plain QMC,
as well as MC.

\textbf{Acknowledgements} We would like to extend our appreciation to Ian Sloan for his valuable comments and insights. We also acknowledge the financial support from the Australian Research Council
for the project DP210100831.

\bibliographystyle{plain}

\end{document}